\documentclass[12pt]{amsart}
 \usepackage{amsmath,amssymb,enumerate,amsfonts,amsthm,graphicx,color}
 \usepackage[all]{xy}
 \usepackage[colorlinks=true, linkcolor=red, linktoc=page, citecolor=blue]{hyperref}
\newtheorem{thm}{Theorem}[section]
\newtheorem{cor}[thm]{Corollary}
\newtheorem{lem}[thm]{Lemma}
\newtheorem{prop}[thm]{Proposition}
 
\theoremstyle{definition}
\newtheorem{defn}[thm]{Definition}
\theoremstyle{definition}
\newtheorem{rem}[thm]{Remark}
\theoremstyle{definition}

\theoremstyle{definition}
 \newtheorem{note}[thm]{Notation}
 \newtheorem{conj}[thm]{Conjecture}
 
  \newcommand{\OO}{\mathcal{O}}
  \newcommand{\PP}{\mathbb{P}}
  \newcommand{\II}{\mathcal{I}}
  \newcommand{\NN}{\mathbb{N}}
  \newcommand{\ZZ}{\mathbb{Z}}
  \newcommand{\St}{\mathbb{S}}
  \newcommand{\X}{\widetilde{X}}
   \newcommand{\R}{\widetilde{R}}
  \newcommand{\T}{\widetilde{T}}
    \newcommand{\Ci}{\widehat{C_i}}
  \newcommand{\al}{\alpha}
   \newcommand{\be}{\beta}
 \newcommand{\ga}{\gamma}
 \newcommand{\de}{\delta}
 \newcommand{\rr}{\bar{r}}
  \newcommand{\qq}{\bar{q}}
  
\begin{document}

\title[Postulation  of lines and
one double line in $\PP^n$]
 {Postulation of generic lines and
 one double line in $\PP^n$  in view of  generic  lines and one multiple linear space}
\author{Tahereh Aladpoosh}
\address{ Tahereh Aladpoosh; School of Mathematics, Institute for Research in Fundamental
Sciences (IPM),  P.O.Box: 19395-5746, Tehran, Iran.}
  \email{tahere.alad@ipm.ir}

\begin{abstract}
A well-known theorem by Hartshorne--Hirschowitz  (\cite{HH})  states
that a generic union $\mathbb{X}\subset \PP^n$, $n\geq 3$, of lines
has good postulation with respect to the linear system
$|\OO_{\PP^n}(d)|$. So a question that arises naturally in studying
the postulation of non-reduced positive dimensional schemes
supported on linear spaces is the question whether adding a
$m$-multiple  linear space $m\PP^r$ to $\mathbb{X}$ can still
preserve it's good postulation, which means in classical language
that, whether $m\PP^r$ imposes independent conditions on the linear
system $|\II_{\mathbb{X}}(d)|$. Recently, the case of $r=0$, i.e.,
the case of lines and one  $m$-multiple point, has been completely
solved by several authors (\cite{CCG4}, \cite{AB}, \cite{B1})
starting with Carlini--Catalisano--Geramita, while the case of $r>0$
was remained unsolved, and this is what we wish to investigate in
this paper. Precisely, we study the postulation of a generic  union
of $s$ lines and one  $m$-multiple  linear space $m\PP^r$ in
$\PP^n$, $n\geq r+2$. Our main purpose is to provide a complete
answer to the question in the case of lines and one double line,
which says that the double line imposes independent conditions on
$|\II_{\mathbb{X}}(d)|$ except for the only case $\{n=4, s=2,
d=2\}$.
 Moreover,  we
discuss an approach to the general case of lines and one
$m$-multiple  linear space,  $(m\geq 2, r\geq 1)$,
 particularly, we find several exceptional such schemes,
and we conjecture that these are the only exceptional ones in this
family. Finally,
 we give some partial results in
 support of our conjecture.
\end{abstract}

\maketitle
 \tableofcontents
\section{Introduction}
 To understand the geometry of a closed subscheme $X$ as an embedded
scheme in $\PP^n$, one of the first
 points of interest is considering the postulation problem, i.e.
 determining the number of conditions imposed by asking
 hypersurfaces of any degree to contain $X$.
 In terms of sheaf cohomology, we would
like to know the
 rank of the restriction maps
 $$\rho(d):  H^0(\PP^n, \OO_{\PP^n}(d))\rightarrow H^0(X, \OO_{X}(d)). $$
 We say that $X$ has  {\emph{maximal rank}} or {\emph{good postulation}} or
 {\emph{expected postulation}} if $\rho(d)$
 has maximal rank, i.e. it is either injective or surjective, for each
 $d\geq 0$. This amounts to saying that one or other of the integers  $h^0(\II_X(d)),~ h^1(\II_X(d))$
 is zero, and this  shows that the property that $X$ imposes independent
 conditions to degree $d$ hypersurfaces can be interpreted
 cohomologically.

On the other hand, this classical problem is equivalent to computing
the Hilbert function of $X$.
 Let $HF(X,d)$ be the Hilbert function of $X$ in
degree $d$, namely,  $HF(X,d)= h^0(\OO_{\PP^n}(d))- h^0(\II_X(d))$,
i.e. the rank of $\rho(d)$.  In order to determine the Hilbert
function of $X$ in some degree $d$, there is an expected value for
it given by a naive count of conditions. This value is determined by
assuming that $X$ imposes independent conditions on the linear
system $|\OO_{\PP^n}(d)|$ of degree $d$ hypersurfaces in $\PP^n$,
i.e.,
$$h^0(\II_X(d))= \max
\left\{h^0(\OO_{\PP^n}(d))-h^0(\OO_X(d)),0\right\},$$
or equivalently
$$HF(X,d)= \min \{HP(X,d),~ HP(\PP^n,d)\},$$
where $HP(X,d)$ is the Hilbert polynomial of $X$.
  In
\cite{CCG1}, the authors called a scheme $X$ with such Hilbert
function for all $d\geq 0$, has {\emph{bipolynomial Hilbert
function}}.
 It always holds $HF(X,d)\leq \min \{HP(X,d),~
HP(\PP^n,d)\},$ so a natural question to ask is: when does this
inequality turn into equality?

 An important observation is that
 the postulation problem depends not only on the numerical data
involved in it, but also
 on the position of the components of $X$.
If $X$ is in sufficiently general position one expects that  $X$ has
bipolynomial Hilbert function, and therefore has good postulation,
but this naive guess is in general false.

When we restrict our attention to the special class of schemes
$X\subset \PP^n$ which supported  on  unions of generic linear
spaces,
 there is much interest in
the postulation problem. In this situation  it is noteworthy that
the notions of good postulation and bipolynomial Hilbert function
coincide.
 Original investigation mostly concentrated on  the reduced cases (see  e.g. \cite{Balm},
 \cite{Bal}, \cite{CCG1},
  \cite{CCG2}, \cite{D}, \cite{Sid}, \cite{HH}, \cite{GO}): if $\dim X=
 0$, i.e. $X$ is a generic collection of points in $\PP^n$, it is well known
 that $X$ has good postulation (see \cite{GMR}); if $\dim X=1$,  we
 have a brilliant result due to Hartshorne and  Hirschowitz,
 going back to $1982$ \cite{HH}, which states that a generic collection of
 lines in $\PP^n, n\geq 3$, has good postulation; as soon as we go up
 to $\dim X> 1$, the postulation problem becomes more and more
 complicated.
The extent of our ignorance in this situation is illustrated by the
fact that the complete answer to the postulation problem even in
two-dimensional case is not yet known (see
 \cite{Bal}, \cite{CCG1} for a generic union of lines and a few
 planes, and  \cite{Balm} for a generic union of lines and a linear space).

On the other hand there is also a lot of interest on the postulation
of non-reduced schemes supported on linear spaces : concerning the
zero-dimensional case, i.e. fat points schemes,
 the postulation problem is a field of  active research in algebraic
 geometry which has occupied researchers's minds
 for over a century, but, despite all the progress made on this problem,
 it is still very live and widely open in its generality
  (see e.g. \cite{AH}, \cite{CH}, \cite{HR}, and also
   \cite{cilib} for a survey of results,  related conjectures and open questions);
  concerning the
 positive dimensional case,
 the postulation problem turns out to be
 far too complicated and giving a complete answer appears to be ambitious
and quite difficult,
 this is why that it has never been systematically studied and in fact it had remained unsolved for a long
 time.
 Apparently the work of Carlini--Catalisano--Geramita \cite{CCG4} is a
 turning point in this story, which together with the recent papers
 \cite{AB} and \cite{B1} shows that
 a generic union $X$ of $s$ lines and one $m$-fat point in $\PP^n, n\geq
 3$,  always has good postulation in degree $d$ except for
 the cases $\{n=3, m= d, 2\leq s\leq d\}$,
 (see for the proof: \cite{CCG4} in the case of $n\geq 4$, \cite{AB} in the case of $\{n=3, m= 3\}$, and \cite{B1}
 in the case of $\{n=3, m\geq 4\}$). As far
as we know, this is the only complete knowledge of the postulation
in the case of non-reduced positive dimensional schemes supported on
linear spaces  (for other related results see  \cite{DHST},
\cite{FHL}).
 This paper was motivated by an attempt to go
further in this direction, namely one may ask about the generic
union of lines and one $m$-multiple line and one may hope that
behaves well with respect to the postulation problem modulo a
certain list of exceptions.
 The main result of this paper is solving the case of $m=2$, i.e.
 the case of
 a generic union of lines and one   double line.
More precisely, we prove the following theorem:
\begin{thm}\label{th1}  Let  $n,d\in\NN$, and $n\geq 3$. Let the scheme $X\subset
\PP^n$ be a generic union of $s\geq 1$ lines and one double line.
Then $X$ has good postulation, i.e.,
$$h^0(\II_X(d)) = \max \left \{ {d+n \choose n}- (nd+1)- s(d+1), 0 \right \},$$
except for the only case $\left\{n=4, s=2, d=2\right\}$.
\end{thm}
 Geometrically, the theorem says that one generic double line in
$\PP^n$ imposes independent conditions to  hypersurfaces of given
degree $d$ containing $s$ generic lines, with the exception of only
the case $\left\{n=4, s=2, d=2\right\}$.

 A first generalization is asking not only for $m=2$, but also
for $m>2$ arbitrary.
 Inspired  by the question involving an arbitrary multiple line instead of a double
 line, and based on several examples, we conjecture that
 a similar result should hold, in analogy with the
 statement of Theorem \ref{th1}.
 Namely we formulate the following conjecture:
 \begin{quote}
\emph{The scheme $X\subset \PP^n$ consisting of $s\geq 1$ generic
lines and one generic $m$-multiple line, $(m\geq 2)$,  always has
good postulation, except for the cases $\left\{n=4, m=d, 2\leq s\leq
d\right\}$.}
 \end{quote}

From this conjecture we can deduce  the important relation of the
 failure of $X$ to have good postulation in degree $d$,  with the
  multiplicity of a multiple line, the dimension
 of the ambient space, and the number of apparent simple lines.
 This seems to be fairly general behavior, which leads us to
 advance in  more general situation.
Indeed, one can push the problem we are facing even further, in the
sense of
 substituting a multiple line with a multiple linear
space of any dimension, and try to make a conjecture which parallels
the above one.

Based on a similar analogy and some further evidence, we
 propose a  conjecture as follows,
  which is significantly stronger than the former.
\begin{conj}\label{con} Let  $n,d,r\in\NN$, and $n\geq r+2\geq 3$. The scheme $X\subset
\PP^n$ consisting of $s\geq 1$ generic lines and one $m$-multiple
linear space of dimension $r$, $(m\geq 2)$,  always has good
postulation, except for the cases
$$\left\{n=r+3, m=d, 2\leq s\leq d\right\}.$$
\end{conj}
 Of course, this conjecture coincides with the previous one for
  $r= 1$.
  As we shall see in \S \ref{scon}, there are
  results that make the conjecture rather plausible.
  It is worth mentioning  that the case $m=1$ is considered  in
  \cite{Balm}, where Ballico proved that a generic disjoint union of lines
  and one linear space always has good postulation \cite[Proposition
  1]{Balm}.

We want to finish by  pointing out that  a non-reduced scheme $X$
supported on  generic union of linear spaces always has exceptions,
a phenomenon that does not happen when $X$ is reduced (according to
a conjecture in \cite{CCG1}). In fact, the ``bad behavior'' of $X$
is always related to the multiple component of it.
\subsubsection*{The structure of the paper}
 Section \ref{bb} contains background material. To be more explicit,
 after recalling  basic definitions and notations on the schemes
of multiple linear spaces in \S \ref{not}, we then give, in \S
\ref{pre}, some lemmas and some elementary observations which are
extremely useful in dealing with the postulation problem. Next, in
\S \ref{res} we collect the known results concerning the postulation
of lines as well as of lines and one multiple point, that are
necessary for our proofs in Sections \ref{s3}--\ref{sn}, in
addition, we look at the Hilbert polynomial of multiple linear
spaces.  We will study the postulation of our schemes by a
degeneration approach, namely,
 degeneration of two skew
 lines in such a way that the resulting degenerated scheme would be a sundial, in the sense of \cite{CCG3},
  thanks to a theorem of  Carlini--Catalisano--Geramita on the postulation of sundials in any
  projective space; all this is
 represented in \S \ref{dege}.

 Section \ref{out} making up the core of the paper devoted to the
 outline of the proof of our main theorem, Theorem \ref{th1}.
 We begin this section with the exceptional case $\{n=4, s=2, d=2\}$ of the
 theorem.
  We make explicit,
 in \S \ref{exe}, the geometric reason that prevents a scheme
 consisting of one generic double line and two generic simple lines
 in $\PP^4$ from imposing independent conditions to quadric
 hypersurfaces. Moreover,  we solve completely the case of $d=2$ of
 the theorem.
 In \S \ref{reph} we explain a rephrasing of Theorem \ref{th1}, that
 is Theorem \ref{th2}. So our goal will be to prove Theorem
 \ref{th1} in the reformulation of Theorem \ref{th2}.
 \S \ref{st} describes in detail our strategy for proving  Theorem \ref{th2}, which is based on
 geometric
 constructions of specialized and degenerated schemes, the
 well-known Horace lemma, and the intersection theory on a
 hyperplane or on a smooth quadric surface.
 We would like to point out that our method of degenerations   owed to
 the works \cite{CCG3,CCG4}.

 In order to apply the strategy we
 will
 use an induction procedure which has difficult but delicate
 beginning steps for $n=3$ and $n=4$.
  In Sections \ref{s3} and \ref{s4} we prove Theorem \ref{th2} for,
  respectively, $n=3$ and $n=4$, setting the stage for our induction
  approach. While, the proof for the general case $n\geq5$ will be
  carried out in Section \ref{sn}.

Conjecture \ref{con}, which geometrically amounts to saying that one
generic $m$-multiple linear space $m\PP^r$ in $\PP^n$, $(n\geq
r+2\geq 3)$,  fails to impose independent conditions to  degree $d$
hypersurfaces through $s$ generic lines  if and only if $\{n=r+3,
m=d, 2\leq s\leq d\}$, is stated and discussed in Section
\ref{scon}, where we prove it for the exceptional cases, and we
describe completely what happen for $d=m$.

Finally, in Appendix, \S \ref{app}, we collect several numerical
lemmas needed for our proofs  in Sections \ref{s4} and \ref{sn}.
\section{Background}\label{bb}
 We  work throughout over an algebraically
closed field $k$ with characteristic zero.
\subsection{Notations}\label{not}
 Given a closed subscheme $X$
of $\PP^n$, $\mathcal{I}_X$ will denote the ideal sheaf of $X$.
 If  $X, Y$ are closed subschemes of $\PP^n$ and $X\subset Y$,
 then we denote by $\mathcal{I}_{X,Y}$ the
 ideal sheaf of $X$ in $\mathcal{O}_Y$.

If $\mathcal{F}$ is a coherent sheaf on the scheme $X$, for any
integer $i\geq 0$ we use $h^i(X,\mathcal{F})$ to denote the
$k$-vector space dimension of the cohomology group
$H^i(X,\mathcal{F})$.
 In particular,
when $X= \PP^n$, we will often omit $X$ and we will simply write
$h^i(\mathcal{F}).$

A \emph{(fat) point of multiplicity $m$}, or an \emph{$m$-multiple
point}, with support $P\in \PP^n$, denoted $mP$, is the
zero-dimensional subscheme of $\PP^n$ defined  by the ideal sheaf
$(\II_P)^m$, i.e. the $(m-1)^{th}$ infinitesimal neighbourhood of
$P$. In case $P\in X$ for any smooth variety $X\subset\PP^n$, we
will write $mP|_{X}$ for the $(m-1)^{th}$ infinitesimal neighborhood
of $P$ in $X$, that is the schematic intersection of the
$m$-multiple point $mP$ of $\PP^n$ and $X$ with $(\II_{P,X})^m$ as
its ideal sheaf.

Similarly, if $L\subset\PP^n$ is a line (resp. linear space), the
closed subscheme of $\PP^n$ supported on $L$ and defined by the
ideal sheaf $(\II_L)^m$ is called a\emph{ (fat) line of multiplicity
$m$} (resp. linear space), or an \emph{ $m$-multiple line} (resp.
linear space), and is denoted by $mL$.

Let $m_1,\ldots,m_s$ be positive integers and let  $X_1,\ldots,X_s$
be $s$ closed subschemes of $\PP^n$.
 We denote
by  $$m_1X_1+ \cdots +m_sX_s$$ the  schematic union of
$m_1X_1,\ldots,m_sX_s$, i.e. the subscheme of $\PP^n$  defined by
the ideal sheaf $(\II_{X_1})^{m_1}\cap \ldots \cap
(\II_{X_s})^{m_s}$.
\subsection{Preliminary lemmas}\label{pre}
The basic tool for the  study of the postulation problem is the so
called \emph{Castelnuovo's inequality}, that is an immediate
consequence of the well-known residual exact sequence (for more
details see e.g. \cite[Section 2]{AH}).

We first recall the notion of residual scheme (\cite[\S
9.2.8]{Ful}).
\begin{defn}
Let $X, Y$ be closed subschemes of $\PP^n$.
\begin{itemize}
\item [(i)]
The closed subscheme of $\PP^n$ defined by the ideal sheaf
$(\II_X:\II_{Y})$ is called the {\textbf{residual} } of $X$ with
respect to $Y$ and denoted by
 $Res_Y(X)$.
\item [(ii)]
The schematic intersection $X\cap Y$ defined by the ideal sheaf
 $(\II_X+\II_{Y})/\II_{Y}$ of $\OO_{Y}$ is called
 the {\textbf{trace} } of $X$ on $Y$ and denoted by  $Tr_Y(X)$.
\end{itemize}
\end{defn}
We note that the generally valid identity for ideal sheaves
$$(\II_{X_1}\cap\II_{X_2}:\II_{Y})=
(\II_{X_1}:\II_{Y})\cap(\II_{X_2}:\II_{Y})$$
 implies that the residual of the schematic union $X_1+X_2$ is the
 schematic union of the residuals.
\begin{lem}[Castelnuovo's Inequality]\label{cas}
 Let $d,e\in \NN$, and $d\geq e$. Let $H\subseteq{\PP^n}$
be a hypersurface of degree $e$, and let $X\subseteq{\PP^n}$ be a
closed subscheme. Then
$$h^0(\PP^n,\II_X(d))\leq
h^0(\PP^n,\II_{Res_H(X)}(d-e))+h^0(H,\II_{Tr_{H}(X)}(d)).$$
\end{lem}
This lemma, especially after the outstanding work of Hirschowitz
\cite{Hir},  is the basis for a standard method of working
inductively with degree to solve the postulation problem and
particularly is central to our proofs in the present paper (Sections
\ref{s3}--\ref{sn}).

The following remark is quite immediate.
\begin{rem}\label{remark}
 Let $n,d,s,s'\in \NN,$  $s'<s.$
  Let $W_s= X_1+\cdots+X_s\subset \PP^n$ be the schematic union of
 non-intersecting closed subschemes $X_i.$
\begin{itemize}
\item[(i)]
 If $h^1(\II_{W_s}(d))= 0,$ then $h^1(\II_{W_{s'}}(d))= 0.$
\item[(ii)]
 If $h^0(\II_{W_{s'}}(d))= 0$, then $h^0(\II_{W_s}(d))= 0.$
 \end{itemize}
\end{rem}
The following  lemma shows that how to add a collection of  points
lying on a linear space $\Pi\subset\PP^n$ to a scheme
$X\subset\PP^n$, in such a way that imposes independent conditions
on the linear system of degree $d$ hypersurfaces passing through $X$
for a given degree d \cite[Lemma 2.2]{CCG1}.
\begin{lem} \label{line}
 Let $d \in \Bbb N$. Let $X  \subseteq \Bbb P^n$ be  a closed subscheme,
 and let $P_1,\dots,P_s$ be generic points on a linear space $\Pi\subset\PP^n$.

 If $h^0(\II_X(d))
=s$ and $h^0(\II_{X +\Pi}(d)) =0$, then
$h^0(\II_{X+P_1+\cdots+P_s}(d)) = 0.$
\end{lem}
\subsection{What results were previously known}\label{res}
As a key question in the direction of studying  the postulation
problem of a scheme  $X\subset \PP^n$ supported on unions of generic
linear spaces, one can ask: \emph{What is the Hilbert polynomial of
$X$?}  When $X$ is reduced, Derksen answered this question by giving
 a  formula for computing the
Hilbert polynomial of $X$ (see \cite{D} for details). Moreover, the
Hilbert polynomial of a multiple linear space is well-known, and it
is not difficult to verify it by a count of parameters, that can be
found in e.g. \cite[\S 2]{B3} and \cite[Lemma 2.1]{DHST}.
\begin{lem}\label{HP}
Let $n,d,r\in \NN$, $r<n$ and $1\leq m\leq d$. Let $\Pi\subset\PP^n$
be a linear space of dimension $r$, then
\begin{equation}\label{00}
HP(m\Pi,d)=\sum_{i=0}^{m-1}{r+d-i \choose r}{n+i-r-1 \choose i}.
\end{equation}
\end{lem}
Indeed, the requirement for a degree $d$ hypersurface in $\PP^n$ to
contain $m\Pi$, i.e. to have multiplicity $m$ along the linear space
$\Pi$, imposes the number of conditions on it, which  is at most the
right hand side of (\ref{00}).

In our case, i.e. the case of double line, one knows that for a
hypersurface to contain a double line $2L$ is equivalent to saying
that it is singular along the line $L$, and   Lemma above asserts
that $2L$ in $\PP^n$  imposes $(nd+1)$ independent conditions to
degree $d$ hypersurfaces.

 Now we recall a few  results on the postulation of
schemes supported on generic linear spaces which we will use to
prove our Theorem \ref{th1} in \S\S \ref{s3}--\ref{sn}. We start
with a spectacular result due to Hartshorne and Hirschowitz, about
the generic lines.
 \begin{thm}[Hartshorne--Hirschowitz]\cite[Theorem 0.1]{HH}\label{HH th}
Let  $n,d\in \NN$, and $n\geq 3$.  Let $X\subset \PP^n$ be a generic
union of $s$ lines. Then $X$ has good postulation, i.e.,
$$h^0(\II_X(d)) = \max \left \{ {d+n \choose n} -s(d+1), 0 \right \}.$$
\end{thm}
As a first step  for positive dimensional non-reduced cases, in
\cite{CCG4}, \cite{AB}, and \cite{B1} the authors  examined the
postulation problem for a generic collection of skew lines and one
fat point in $\PP^n$, and they found out that when $n\geq 4$ these
schemes have good postulation, but when $n= 3$
 there are several defective such schemes.
Now,  one can present these results  simultaneously in a theorem as
follows.
\begin{thm}
Let  $n,m,d\in \NN$, and $n\geq 3$. Let the scheme $X\subset \PP^n$
be a generic union of $s\geq 1$  lines and one fat point of
multiplicity $m\geq 2$. Then $X$ has good postulation, i.e.,
$$h^0(\II_X(d))= \max \left\{{d+n\choose n}- {m+n-1\choose n}-
s(d+1), 0\right\},$$ except for the cases $\{n=3, m=d, 2\leq s\leq
d\}.$
\end{thm}
 Since we will apply the theorem for the case $m=2$ and $d\geq 3$ several times in the next sections,
 it is convenient to restate it as follows.
 \begin{cor} \label{CCG}
Let  $n,d\in \NN$, and  $n,d\geq3$. Let the scheme $X\subset \PP^n$
be a generic union of $s\geq 1$  lines and one double point. Then
$X$ has good postulation in degree $d$, i.e.
$$
h^0(\II_X(d)) = \max \left \{ {d+n \choose n} -  (n+1) -s(d+1), 0
\right \}.
$$
\end{cor}
\subsection{A degeneration approach}\label{dege}
A natural approach to the postulation problem is to argue by
degeneration.
 In view of  the fact that we have the semicontinuity
theorem for cohomology
 groups in a flat family, one may use the degenerations and the
 semicontinuity theorem in order to be able to better handle
 the
postulation of generic configuration of linear spaces. Specifically,
if one can prove that the property of having good postulation is
satisfied in the special fiber, i.e. the degenerate scheme, then one
may hope to  obtain the same property in the general fiber, i.e. the
original scheme.

In the celebrated paper \cite{HH} Hartshorne and Hirschowitz
investigated a new degeneration technique to attack the postulation
problem for a generic union of lines.  In fact,  they degenerate two
skew lines in $\PP^3$ in such a way that the resulting scheme
becomes  a ``degenerate conic with an embedded point" (which also
was used in \cite{Hir}).
 Even more generally, one can push this trick of ``adding nilpotents" further,
  to give a degeneration of two skew lines in higher dimensional
  projective space $\PP^n,$ $n\geq 3$,   this is what
 the authors introduced in \cite[Definition 2.7]{CCG1} and called a
  {\em (3-dimensional) sundial}.

 According to the terminology of \cite{HH}, we say that $C$ is a
{\it degenerate conic} if  $C$ is the union of two intersecting
lines $L$ and $M$, so $C=L+M$.
\begin {defn}\label{sundial}
Let $L$ and $M$ be two intersecting lines in $\mathbb P^n, ~n \geq
3$, and let
 $T\cong \Bbb P^{3}$ be a generic linear space
containing the degenerate conic $L+M$. Let
 $P$ be the singular point of  $L+M$,
i.e. $P = L \cap M$.   We call the scheme $L+M+ 2P|_T$ a {\it
degenerate conic with an embedded point} or a {\it (3-dimensional)
sundial}.
\end {defn}
 One can show a sundial is a flat limit inside $\PP^n$ of a flat
 family whose general fiber is the disjoint union of two lines, i.e.
 a sundial is a degeneration of two generic lines in $\PP^n$, $n\geq 3$. This is the
 content of the following lemma (see \cite[Example 2.1.1]{HH} for the case $n=3$, and
  \cite[Lemma 2.5]{CCG1} for the general case).
\begin{lem}
Let $X_1\subset\PP^n$, $n\geq 3$, be the disjoint union of two lines
$L_1$ and $M$. Then there exists a flat family of subschemes
$X_i\subset \langle X_1\rangle\cong \PP^3$, $(i\in k)$,  whose
general fiber is
 the union of two skew lines and whose special fiber is the sundial $X_0= M+L+2P|_{\langle X_1\rangle}$,
 where $L$ is a line and $M\cap L= P$.
\end{lem}
Note that we can also easily degenerate a simple generic point and a
degenerate conic to a sundial. Therefore, a sundial is either a
degeneration of two generic lines, or a degeneration of a scheme
which is the union of a degenerate conic and a simple generic point.

We recall here the main result in \cite{CCG3}  which guarantees that
a generic collection of sundials will behave well with respect to
the postulation problem.
\begin{thm}\cite[Theorem 4.4]{CCG3}\label{sundial}
Let $n,d\in \NN$, and $n\geq 3$. Let the scheme $X\subset \PP^n$ be
a generic union of  $x$   sundials and $y$ lines.  Then $X$ has good
postulation, i.e.
$$h^0(\II_X(d))= \max\left\{ {d+n\choose n}- (2x+y)(d+1), 0\right\}. $$
\end{thm}
\section{Outline of the proof of Theorem \ref{th1}}\label{out}
Now we have all the necessary tools to tackle our main theorem,
Theorem \ref{th1}.
\subsection{The Exceptional Case}\label{exe}
We look for the case where $X$ fails to have good postulation.
 Actually, there is only one exception in this infinite family,
 namely the case $\{n=4, s=2\}$ which,  $H^0(\II_X(2))$
has dimension one instead of zero. As we will see below, this
exceptional case arises from geometric reason, although the proof
follows from numerical reason.

 Now we prove the following proposition, which completely describes the case $d=2$ of
 Theorem \ref{th1}:
\begin{prop}\label{d2}
 The scheme  $X\subset \PP^n, n\geq 3,$
consisting of $s\geq 1$ generic lines  and one double line $2L$ has
good postulation in degree $d=2$, i.e.,
\begin{eqnarray*}
h^0(\II_X(2))&=&\max \left \{ {n+2 \choose n}  -(2n+1)- 3s, 0 \right
\} \\
&=&\max \left\{ {n \choose 2} -3s, 0 \right\},
\end{eqnarray*}
except for the only case $\left\{n=4, s=2\right\}$.
\end{prop}
\begin{proof}
 The sections of $\II_X(2)$ correspond to quadric hypersurfaces
 in $\PP^n$ which, in order to contain $2L$, have to be cones whose
 vertex contains the line $L$.

If $n=3$, we obviously have $h^0(\II_X(2))=0,$ as expected.

If $n\geq 4$, we consider the projection $X'$ of $X$ from $L$ onto a
generic linear subspace
 $\PP^{n-2}\subset \PP^n$, hence $X'$ is a scheme consisting of
 $s$ generic lines in $\PP^{n-2}$.
 It follows that
 $$h^0(\PP^n,\II_X(2))= h^0(\PP^{n-2},\II_{X'}(2)).$$

  In case $n>4$, by Hartshorne--Hirschowitz theorem (Theorem \ref{HH th}) we get
$$h^0(\PP^{n-2},\II_{X'}(2))= \max \left\{{2+n-2 \choose 2} -3s,
0\right\}= \max\left\{ {n \choose 2} -3s, 0\right\},$$ and we get
the conclusion.

In case $n=4$, $X'$ is a generic union of $s$ lines in $\PP^2$.
  Hence, for $s>2$ it is immediate to see that $h^0(\II_{X'}(2))=
  0$.
 For $s\leq 2$ we have $h^0(\II_{X'}(2))= {2-s+2 \choose 2}= {4-s
 \choose 2}$, on the other hand the expected value for
 $h^0(\II_{X}(2))$ is  $\max\left\{ 6 -3s, 0\right\}$.
 Thus for $s=1$,  $h^0(\II_{X}(2))= 3,$  as expected; but  for
 $s= 2$, $h^0(\II_{X}(2))= 1$ while the expected
 one is $0$, which is what we wanted to show.
\end{proof}
\subsection{Rephrasing  Theorem \ref{th1}}\label{reph}
From what we have observed in the previous subsection, it remains to
verify Theorem \ref{th1} for $d\geq 3$,  asserts that schemes
$X\subset \PP^n$ consisting of $s$ generic lines and one generic
double line have good postulation for all $d\geq 3$, i.e.,
$h^0(\II_X(d))=0$ or $h^1(\II_X(d))=0.$ (Note that the case $d=1$ of
Theorem \ref{th1} being trivial, so we omit it.)

First note that,
 as $X$ varies in a flat family, by the semicontinuity of cohomology,
the condition of good postulation, is clearly an open condition on
the family of $X$. Hence to prove Theorem \ref{th1}, it is enough to
find any scheme  of $s$ lines and one double line, or even any
scheme which is a specialization of a flat family of $s$ lines and
one double line, which has good postulation.

Given $n$ and $d$, suppose one can choose $s$ so that:
$${d+n \choose n}= (nd+1)+ s(d+1)$$
and suppose  one can find $X$ so that  $h^0(\II_X(d))=
h^1(\II_X(d))= 0$. Then if one removes some lines from $X$, one gets
a scheme $X'\subset X$ such that $h^1(\II_{X'}(d))$ will still be
zero; and if one adds some lines to $X$, one gets a scheme
$X''\supset X$ such that $h^0(\II_{X''}(d))$ will still be zero (by
Remark \ref{remark}). In other words, the good postulation for that
given $n, d$ and the unique integer $s$, that gives the equality
above, implies the good postulation for the same $n, d$ and any $s$
whatsoever.

Unfortunately for given $n, d$ one can not always find such an $s$.
Therefore we will make adjustments by adding some collinear  points
to $X$, to get a similar equality. In particular,  we prove the
following theorem which,   by Remark \ref{remark}, implies our main
Theorem \ref{th1}:
\begin{thm}\label{th2}
 Let  $n,d\in\NN$, and $n,d\geq
3$. Let
$$r= \left \lfloor {{d+n \choose n} -  (nd+1)}\over {d+1}\right \rfloor
 ; \ \ \ \
q={d+n \choose n} - (nd+1) - r(d+1).
 $$
Let the scheme  $X\subset \PP^n$ be a generic union of $r$ lines
$L_1,\ldots,L_r$, one double line $2L$ and $q$  points
$P_1,\ldots,P_q$ lying on a generic line $M$.  Then $X$ has good
postulation, i.e.,
$$h^1(\II_X(d))= h^0(\II_X(d)) =   {d+n \choose n} -(nd+1) - r(d+1)- q= 0.$$
\end{thm}
From our discussion above,  Theorem \ref{th1} follows immediately
from this theorem. Indeed, to prove Theorem \ref{th1} for that $n,
d$ and any $r'\leq r$, simply remove the $q$ points and $r-r'$
lines, then the corresponding $h^1(\II(d))$ will be zero; to prove
it for $r''>r$, first add the line $M$ passing through the $q$
collinear points, then add $r''-r-1$ disjoint lines, then the
corresponding $h^0(\II(d))$ will be zero.
\begin{note}
We denote  by  $\St(n,d)$ and $\St^*(n,d)$ for $d\geq 3$,  the
statement of Theorem \ref{th1} and the statement of Theorem
\ref{th2}, respectively.
\end{note}
\begin{rem}
As we   have seen, the statement  $\St^*(n,d)$ implies $\St(n,d)$.
On the other hand, the converse also follows directly from  Lemma
\ref{line}.
\end{rem}
\subsection{Strategy of the Proof}\label{st}
We illustrate our general strategy explicitly to prove  Theorem
\ref{th2}  for a generic scheme $X\subset\PP^n$ consisting of one
double line, $r$ simple lines and $q$ collinear points,  as follows:
The difficulty with proving a property like ``good postulation" is
that, it is very hard to lay hands on a
 \emph{generic} scheme $X$. Our approach to overcome to this difficulty is
 to start with a special scheme, which is obtained by
 several different kind of specializations and degenerations,  and then use
  semicontinuity theorem for cohomology groups to discover the
same  property for  generic scheme $X$.
 The next step is to  reduce the postulation problem of our scheme,
  via Castelnuovo's inequality, to
 the study of the postulation of a  residual
scheme  and a trace scheme, that is
 \emph{La m\'ethode d'Horace}, elaborated by A. Hirschowitz \cite{Hir}.

 To be more precise,  for $n\geq 4$ we
 specialize $x$ simple lines into a fixed hyperplane $\PP^{n-1}\subset \PP^n$
  and degenerate $q'$ other  pairs of
 simple lines to sundials,  further, we specialize these sundials into $\PP^{n-1}$
 unless their singular points,  which requires
 a capability of guessing the right
specialization.
 Thus if one can choose these numbers correctly, that is in such a way that  the numbers  $x$  and $q'$
  are sufficiently many  to comply with the
induction hypothesis on degree (see Appendix, Lemma \ref{ap1}), then
the residual has good postulation, while the trace is a scheme in
  $\PP^{n-1}$,  which is more complicated to verify  because of the appearance of $q'$  degenerate conics and one
  double point.
 Now to handle the problem involving the postulation of the trace scheme we specialize $\rr$ lines,
 $\qq$  simple points,  and
  the double point into a fixed hyperplane $\PP^{n-2}\subset \PP^{n-1}$, then
  we take again residual and trace. Of course, the numbers  $\rr$ and $\qq$
    should not be too numerous, and we have to find
   these numbers  satisfying all the necessary
  inequalities (see Appendix, Lemma \ref{ap2}).
   This time the trace  consists of $\rr$ lines, some simple points, and
   one double point, which by Corollary \ref{CCG} has good
   postulation, while the residual  consists of $q'$ degenerate
   conics, $(x-\rr)$ lines and some simple points, which we will degenerate it  to a scheme
   consisting of
  $q'$ sundials, $(x-\rr)$ lines and some points, that by Theorem \ref{sundial} has good
   postulation (these arrangements contain a lot of technical
   details which can be found in Appendix, Lemma \ref{ap2}).

This argument for the trace of $X$  can be applied in  cases $n\geq
5$, but unfortunately does not cover the case $n=4$, where forced
intersection of lines appear in $\PP^{n-2}= \PP^2$. In fact the case
$n=4$ will be taken care of by
 a smooth quadric surface $Q$
 and a way of specialization which is considerably different from that mentioned
above. Explicitly, we specialize one line of each of the degenerate
conics, together with $\hat{r}$ simple lines, into the same ruling
on $Q$, moreover, we specialize $\hat{q}$ simple points onto $Q$,
then we take again residual and trace. Surely, the numbers $\hat{r}$
and $\hat{q}$ should not be so much, and we have to find these
numbers satisfying all the necessary inequalities (see Appendix,
Lemma \ref{ap3}).
 Now the current residual
consisting
 of one double point, $(x-\hat{r}+q')$ lines and some simple points will be verified by Corollary
\ref{CCG}, while the current trace, which is a scheme in $Q$,  will
be verified by applying some results from internal geometry  of $Q$.

{\em What about in $\PP^3$?}
 Actually, the most difficult part of the proof is the case of $\PP^3$.
 Our approach to this case uses extremely an ad hoc method
 which is done via specializing as many lines as is needed
 into a smooth quadric surface instead of a plane, and then, if necessary,    degenerating
 some other pairs of simple lines to sundials (and even more
 specializing
  sundials and points), that
 requires a case by case discussion.
  Here the role of the smooth quadric is explained by the property
 of having two rulings of skew lines and by the known results from intersection theory on it.
 We notice that also  in the case of $\PP^3$ our method can then be
 applied under certain numerical conditions, and this is why
   the proof splits
 into three specific cases $d\equiv 0$ (mod 3), $d\equiv 1$ (mod 3)
 and $d\equiv 2$ (mod 3), which described exactly in Section \ref{s3}.
 In fact, our method can be safely applied for  $d\equiv 0$ (mod
 3), as well as for $d\equiv 1$ (mod 3), but a slight complication
 arises in the case of
 $d\equiv 2$ (mod
 3), where we have to consider a different specializaton, which is done by
 placing the support of the double line into the smooth quadric.

Summing up,   the method for proving our Theorem \ref{th2}, based on
the  induction on degree $d$,  breaks down into three  parts: $n=3$,
$n=4$, and $n\geq 5$, which we have to investigate each of them
separately in \S\S \ref{s3}--\ref{sn}.

We would like to point out that to make the strategy applicable,
many verifications are needed because of the messy arithmetic
involved (see \S \ref{app}).

Since to prove the property of good postulation, according to our
strategy,   we will  use in the sequel  Castelnuovo's inequality and
the semicontinuity of cohomology several times, it will be useful to
consider the following remark.
\begin{rem}\label{rem}
With the hypotheses of Theorem \ref{th2}, let
 $\widetilde{X}$ be the scheme obtained from $X$ by combining
 specializations and degenerations via a fixed hypersurface $H\subset \PP^n$ of
 degree $e$.

 If $h^0(\II_{Res_{H}(\X)}(d-e))= 0$ and $h^0(H,
 \II_{Tr_{H}(\X)}(d))=0$, then by Castelnuovo's inequality (Lemma
 \ref{cas}) we have $h^0(\II_{\X}(d))= 0$, and this implies, by the
 semicontinuity of  cohomology, $h^0(\II_X(d))= 0$.
\end{rem}
\section{Proof  in $\PP^3$}\label{s3}
In this section we  prove Theorem \ref{th2} in $\PP^3$.

 We start with a useful observation concerning
the behaviour of certain one-dimensional subschemes  of a smooth
quadric surface $Q\cong \PP^1\times\PP^1$ with respect to the linear
system of curves of type $(a,b)$, which we will often use in the
sequel (for a  proof  see \cite[Lemma 2.3]{HH}).
\begin{lem}\label{hh}
Let $\alpha,\beta,\gamma,\delta, d\in \NN$, and let $Q\subset \PP^3$
be a smooth quadric. Let the scheme $W\subset Q$ be a generic union
of $\alpha$ lines belonging to the same ruling of $Q$,
  $\beta$ simple  points,
 $\gamma$
 simple points lying on a line belonging to the same ruling of the $\alpha$ lines,
and $\delta$ double points. If the following conditions are
satisfied:
\begin{itemize}
\item[(1)] $\alpha(d+1)+ \beta+ \gamma+ 3\delta= (d+1)^2$;
\item[(2)] $\delta\leq d+1$;
\item[(3)] $\gamma\leq d+1$;
\item[(4)] if $d>\alpha$ then $\delta\leq \frac{d+1-\gamma}{2}+ (d-\alpha-1)\left\lfloor
\frac{d+1}{2}\right\rfloor$, otherwise $\delta= 0$;
\end{itemize}
then $h^1(Q,\II_W(d))= h^0(Q,\II_W(d))= 0$.
\end{lem}
Before we begin our investigations in the case of $\PP^3$,
 we recall some elementary facts about the geometry on the smooth  quadric surface $Q$:
 the divisor class group of $Q$ is $\ZZ\oplus\ZZ$, generated by a
 line in each of the two rulings; by the type we mean the class in
 $\ZZ\oplus \ZZ$; the curves on $Q$ are those of type $(a,b)$ with
 $a,b\geq 0$; by convention
 $\OO_Q(d)= \OO_Q(d,d)$; finally $h^0(Q,\OO_Q(a,b))= (a+1)(b+1)$.

Now we state and prove Theorem \ref{th2} in $\PP^3$, that is:

\vspace{0.2cm}
 $\St^*(3,d)$:
  Let $d\geq 3$ and
$$r= \left \lfloor {{d+3 \choose 3} -  (3d+1)}\over {d+1}\right \rfloor
 ; \ \ \ \
q={d+3 \choose 3} - (3d+1) - r(d+1).
 $$
Let the scheme  $X\subset \PP^3$ be a generic union of $r$ lines
$L_1,\ldots,L_r$, one double line $2L$ and $q$  points
$P_1,\ldots,P_q$ lying on a generic line $M$.  Then $X$ has good
postulation, i.e.,
$$h^1(\II_X(d))= h^0(\II_X(d)) =   {d+3 \choose 3} -(3d+1) - r(d+1)- q= 0.$$
\begin{proof}
In order to start the induction argument we need to establish the
base cases $d=3, 4$.

First consider the case $d= 3$. In this case we have $r= 2$ and
 $q= 2$, therefore $X= 2L+ L_1+ L_2+ P_1+ P_2\subset \PP^3.$

 Fix a generic plane $H\subset \PP^3$, and consider the scheme
 $\widetilde{X}$  obtained from $X$ by specializing
  the line $L_1$ and the points $P_1, P_2$ into $H$. By abuse of
  notation, we will  again denote these specialized line and points by $L_1$ and $P_1,P_2$.
  (Keeping in mind that in the sequel, by abuse of notation, we will
  always
  denote the specialized components by the same letters as the
  original ones.)

We have $Res_H(\widetilde{X})= 2L+ L_2$, then  it is obvious that
$$h^0(\II_{Res_H(\widetilde{X})}(2))= 0.$$

 Also,
 $Tr_H(\widetilde{X})= 2R|_H+ L_1+ S+  P_1+ P_2\subset H$,
  where $L\cap H= R$ and $L_2\cap H= S$.
Since $L_1$ is a fixed component for the  sections of
 $\II_{Tr_H(\widetilde{X})}(3)$, we get that
 $$h^0(H,\II_{Tr_H(\widetilde{X})}(3))= h^0(H,\II_{Tr_H(\widetilde{X})- L_1}(2)).$$
Since the  double point $2R|_H$  imposes $3$ independent  conditions
on $|\OO_H(2)|$,
  and the points $P_1, P_2, S$ are generic points in $H$, we easily get
 that
 $$h^0(H,\II_{Tr_H(\widetilde{X})- L_1}(2))= {2+2 \choose 2}-3-3= 0.$$
  Thus  by Remark \ref{rem} we have $h^0(\II_X(3))= 0$, that is,
 $X$ has good postulation in degree $3$.

Now consider the case $d= 4$. We have $r= 4$ and $q= 2$, then $X$ is
the schematic union:
 $X= 2L+ L_1+ L_2+ L_3+ L_4+ P_1+ P_2\subset \PP^3$.

Let $Q$ be a smooth quadric surface, and let $\widetilde{X}$ be the
scheme obtained from $X$ by specializing three of the lines $L_i$ in
such a way that $L_1,L_2,L_3$ become lines of the same ruling on
$Q$, and by specializing the points $P_1,P_2$ onto $Q$.

Then we get $Res_Q(\widetilde{X})= 2L+ L_4\subset \PP^3,$ and it
clearly follows that  $$h^0(\II_{Res_Q(\widetilde{X})}(2))= 0.$$

Consider the trace of $\widetilde{X}$ on $Q$, that is
 $$Tr_Q(\widetilde{X})= 2R_1|_Q+ 2R_2|_Q+ L_1+ L_2+ L_3+ S_1+
S_2+ P_1+ P_2\subset Q,$$
 where $L\cap Q= R_1+ R_2$ and $L_4\cap Q= S_1+ S_2$.
 Note that the scheme $Tr_Q(\widetilde{X})$ is  generic union in $Q$  of
 three lines belonging to the same ruling of $Q$, four simple points and
 two double points, hence we can apply Lemma \ref{hh}, with $(\al= 3, \be= 4, \ga= 0, \de= 2, d=4)$,    and we obtain
$$h^0(Q,\II_{Tr_Q(\widetilde{X})}(4))= 0.$$
 So by Remark \ref{rem} it follows that
 $h^0(\II_X(4))= 0$. Hence the case $d=4$ is done.

Now assume $d\geq 5$. We consider three cases, and we proceed by
induction on $d$. Let $Q$ be a smooth quadric surface in $\PP^3$.

\vspace{0.2cm}
{\textbf{Case $d\equiv 0$ (mod 3).}}
 Write $d= 3t, t\geq 2$. Then
  $$r= \frac{(t+1)(3t+2)}{2}- 3,\ \ \ \ \ \ q= 2.$$
  We have $X= 2L+ L_1+\cdots+ L_r+ P_1+P_2\subset \PP^3$.
  Since $2t+1 \leq r$, we specialize $2t+1$ of the lines $L_i$
  in such a  way that  $L_1,\ldots,L_{2t+1}$ become lines of the
  same ruling on $Q$, and we denote by $\widetilde{X}$ the specialized
  scheme. We have
  $Res_Q(\widetilde{X})= 2L+ L_{2t+2}+\cdots+L_r+ P_1+ P_2\subset \PP^3,$
  which is the generic union of one double line, $\frac{t(3t+1)}{2}-
  3$ lines and two points,
  so by the induction  hypothesis it follows that
   $$h^0(\II_{Res_Q(\widetilde{X})}(d-2))= 0.$$
Now we treat the trace scheme
\begin{eqnarray*}
Tr_Q(\widetilde{X})&=&2R_1|_Q+ 2R_2|_Q+ L_1+\cdots+ L_{2t+1} \\
& &+ S_{1,2t+2}+ S_{2,2t+2}+ \cdots+ S_{1,r}+S_{2,r}\subset Q,
 \end{eqnarray*}
 where $L\cap Q= R_1+ R_2$ and $L_i\cap Q= S_{1,i}+ S_{2,i}$, $(2t+2\leq i\leq
 r)$. Note that the points $R_1, R_2, S_{1,i}, S_{2,i}, (2t+2\leq i\leq r),$ are generic points on $Q$.
That is  $Tr_Q(\widetilde{X})$ consists of $2t+1$ lines of the same
 ruling on $Q$, two generic double points and $t(3t+1)- 6$ generic simple points,
 then we can easily check that  $Tr_Q(\widetilde{X})$ satisfies the conditions of Lemma \ref{hh},
 with $(\al= 2t+1, \be= t(3t+1)- 6, \ga= 0, \de= 2)$,
 and this implies
  $$h^0(Q, \II_{Tr_Q(\widetilde{X})}(d))= 0.$$
  Hence by Remark \ref{rem}we get
  $h^0(\II_X(d))= 0$.

\vspace{0.2cm}
{\textbf{Case $d\equiv 1$ (mod 3).}}
 Write $d= 3t+1, t\geq 2$. Then
  $$r= \frac{(t+1)(3t+4)}{2}- 3,\ \ \ \ \ \ q= 2.$$
In this case we have $X= 2L+ L_1+\cdots+ L_{r}+ P_1+ P_2\subset
\PP^3.$

  We wish to construct a specialization of $X$
  so that the expected vanishing  $h^0(\II_X(d))=0$ is
 obtained. In order to do this, we introduce the specialization
 $\X$
 of $X$ in the following way:
\begin{itemize}
\item   specialize
   the points $P_1,P_2$  onto $Q$;
\item  specialize the first $2t+1$  lines $L_i$  in such a  way that
 they become lines of the
  same ruling on $Q$, and call the resulting set of lines $X_1$;
\item
  degenerate the next $2t-2$ pairs of  lines $L_i$, so that they become
   $2t-2$ sundials $\widehat{C_i}= C_i+ 2N_i$, $(1\leq i\leq 2t-2)$,  where $C_i$ is a
   degenerate conic
    and $2N_i$ is a double point with support at the singular point of $C_i$, furtheremore,
specialize
   the points $N_1,\ldots,N_{2t-2}$ onto $Q$,
   and call the
   resulting scheme of sundials  $X_2$, that is $$X_2= \widehat{C_1}+\cdots+ \widehat{C_{2t-2}},$$ with the
   property that the singular points of $\Ci$ lie on $Q$;
    \item leave the remaining simple lines $L_i$, which are   $r- (2t+1)- 2(2t-2)= \frac{t(3t-5)}{2}+2$ lines,
    generic outside $Q$, and call this collection of lines
     $X_3$;
\end{itemize}
 notice that we can do the above specialization because of the inequality  $r\geq 2t+1+
 2(2t-2)$.
 Then by letting  $$\X= 2L+ X_1+ X_2+ X_3+ P_1+ P_2\subset \PP^3,$$
  we  get the desired specialization of $X$.

 Now we perform the process of verifying the residual and the trace
 on this specialized scheme $\X$.
We obtain
 \begin{eqnarray*}
  Tr_Q(\X)&=& 2R_1|_Q+ 2R_2|_Q+ X_1+ Tr_Q(X_2) \\
  & &+ Tr_Q(X_3)+ P_1+
 P_2\subset Q,
 \end{eqnarray*}
where $L\cap Q= R_1+ R_2$;
 $Tr_Q(\Ci)= 2N_i|_Q+ C_i\cap Q$, and $C_i\cap Q$ is a union of two simple
 points,  $(1\leq i\leq 2t-2)$, therefor  $Tr_Q(X_2)$ consists of $2t-2$ double points and $4t-4$
 simple points; moreover,
 $Tr_Q(X_3)$ consists of $t(3t-5)+ 4$ simple points.
 Hence $Tr_Q(\X)$ is generic union in $Q$  of $2t+1$ lines belonging to the same ruling
 of $Q$,  $2t$ double points and $t(3t-1)+ 2$ simple points. An easy computation, yields
 that the scheme $Tr_Q(\X)$
 verifies the conditions of Lemma \ref{hh}, with $(\al= 2t+1, \be= t(3t-1)+ 2, \ga= 0, \de= 2t)$,  then we have
 $$h^0(Q, \II_{Tr_Q(\X)}(d))= 0.$$ So we are done with $Tr_Q(\X)$.
If we can prove  $h^0(\II_{Res_Q(\X)}(d-2))= 0$ then, by
Castelnuovo's inequality, we get $h^0(\II_{\X}(d))= 0$.

 Here we consider the residual scheme
 $$Res_Q(\X)= 2L+ C_1+ \cdots+ C_{2t-2}+ X_3\subset \PP^3.$$
 In order to compute $h^0(\II_{Res_Q(\X)}(d-2))$, we need to construct a
 specialization of $Res_Q(\X)$, and take again the residual and the trace with
 respect to $Q$.

 First, let $M_{1,i}, M_{2,i}$ be the two lines which form the degenerate
 conic $C_i$, $(1\leq i\leq 2t-2)$, this means $C_i= M_{1,i}+
 M_{2,i}$. Pick a line $L'\subset X_3.$
 Now let $\R$ be the scheme obtained from $Res_Q(\X)$ by specializing the
 degenerate conics $C_i$ and  the lines $L, L'$
 in such a way that the lines
 $M_{1,1}\ldots,M_{1,2t-2}$ and $L, L'$ become $2t$ lines of the same ruling
 on $Q$ (the lines $M_{2,1}\ldots,M_{2,2t-2}$ and the other $\frac{t(3t-5)}{2}+ 1$ lines of $X_3$ remain generic lines,
 not lying on $Q$).

 From this specialization we have
$$Res_Q(\R)= L+ M_{2,1}+\cdots+ M_{2,2t-2}+ (X_3-L'),$$ that is  generic
union of $\frac{t(3t-1)}{2}$ lines in $\PP^3$, hence by
Hartshorne--Hirschowitz theorem (Theorem \ref{HH th}) we immediately
get, (note that $d= 3t+ 1$),
$$h^0(\II_{Res_Q(\R)}(d-4))= { {d-4+3}\choose{3}}- \frac{t(3t-1)}{2}(d-4+1)= 0.$$
On the other hand,
 $M_{2,i}$ meets $Q$ in the two points which are $M_{1,i}\cap
 M_{2,i}$, that is contained in $M_{1,i}$,  and another point, which we denote by $S_i$.
 Thus
\begin{eqnarray*}
 Tr_Q(\R)&=& 2L|_Q+ M_{1,1}+\cdots+ M_{1,2t-2}+ L' \\
& &+ S_1+\cdots+ S_{2t-2}+ Tr_Q(X_3-L')\subset Q,
 \end{eqnarray*}
 where $Tr_Q(X_3-L')$ is
made by $t(3t-5)+2$ generic points. Therefore the scheme $Tr_Q(\R)$
is  generic union in $Q$ of one double line, $2t-1$ lines, such that
all of these $2t$ lines are placed in the same ruling of $Q$, and
$3t(t-1)$ points. Considering $Q$ as $\PP^1\times \PP^1$ and
assuming these $2t$ lines belong to the first ruling of $Q$, we see
that  each of these lines is a curve of type $(1,0)$ on $Q$.

 Note that the double line  $2L|_Q$
and the lines $M_{1,i}, L'$, $(1\leq i\leq 2t-2)$, are fixed
components for the curves  of $H^0(Q,\II_{Tr_Q(\R)}(d-2,d-2))$,
since $d-2\geq 2t+1$. Now set $\Lambda= 2L|_Q+ M_{1,1}+\cdots+
M_{1,2t-2}+ L'\subset Q$, which is of type $(2t+1,0)$. Hence by
removing the fixed component $\Lambda$, and by using the fact that
the scheme $Tr_Q(\R)- \Lambda$ is  generic union of $3t(t-1)$ simple
points, moreover by recalling the equality $d= 3t+1$, we deduce
\begin{eqnarray*} h^0(Q,
\II_{Tr_Q(\R)}(d-2, d-2))&=& h^0(Q,
\II_{Tr_Q(\R)- \Lambda}(d-2-(2t+1),d-2)) \\
 &=& h^0(Q,\II_{Tr_Q(\R)-\Lambda}(t-2,3t-1)) \\
 &=&  h^0(Q,\OO_Q(t-2,3t-1))- 3t(t-1) \\
 &=& (t-1)3t- 3t(t-1)=0.
\end{eqnarray*}

This together with $h^0(\II_{Res_Q(\R)}(d-4))= 0$, implies that
$$h^0(\II_{\R}(d-2))= 0,$$ consequently, by semicontinuity,
$h^0(\II_{Res_Q(\X)}(d-2))= 0.$ So we conclude that
$h^0(\II_{\X}(d))= 0$, and from here, by Remark \ref{rem}, we get
$h^0(\II_X(d))= 0$.

 \vspace{0.2cm}
{\textbf{Case $d\equiv 2$ (mod 3).}}
 Write $d= 3t+2, t\geq 1$. Then
  $$r= \frac{3t(t+3)}{2},\ \ \ \ \ \ q= t+3.$$
  We have $X= 2L+ L_1+\cdots+ L_r+ P_1+\cdots+ P_{t+3}\subset \PP^3,$ where $P_1,\ldots, P_{t+3}$ are
  points lying on a generic line $M$.

Realize  $Q$ as $\PP^1\times \PP^1$.
 We  specialize $2t$ of the lines $L_i$ and the lines $L, M$
  in such a  way that  $L_1,\ldots,L_{2t}$ and $L, M$ become $2t+2$ lines of the
  first ruling on $Q$, i.e. each  has type $(1,0)$,
   and  we denote by $\widetilde{X}$ the specialized
  scheme (note that this is possible since $r\geq 2t+2$).
  It is clear   from this specialization that the points
  $P_1,\ldots,P_{t+3}$ become points on the line $M$ belonging to the
  first
  ruling of  $Q$.

First we consider the residual scheme
 $$Res_Q(\X)= L+ L_{2t+1},\cdots+ L_r\subset \PP^3,$$ that is  generic
 union of $r-2t+1= \frac{(t+1)(3t+2)}{2}$ lines, so according to
 Hartshorne--Hirschowitz theorem we get
 $$h^0(\II_{Res_Q(\X)}(d-2))= { d-2+3 \choose 3}-
 \frac{(t+1)(3t+2)}{2} (d-2+1)= 0.$$
Then we are left with the trace scheme, which is
 $$Tr_Q(\X)= 2L|_Q+ L_1+\cdots+ L_{2t}+ X_1+ P_1+\cdots+
 P_{t+3}\subset Q,$$
where $X_1= Tr_Q(L_{2t+1}+\cdots+ L_r)$. Using the fact that each
$L_i$  meats $Q$ at two  points,  $(2t+1\leq i\leq r)$, it follows
that $X_1$ is made by $2(r-2t)= t(3t+5)$ simple points.

Observe that the double line  $2L|_Q$ and the lines
$L_1,\ldots,L_{2t}$ are fixed components for the curves of
$H^0(Q,\II_{Tr_Q(\X)}(d,d))$, (note that $d\geq 2t+2$). Set
$\Lambda= 2L|_Q+ L_1+\cdots+ L_{2t}\subset Q$, which has type
$(2t+2,0)$.
 Removing the fixed component $\Lambda$ implies that
 \begin{eqnarray*}
  h^0(Q,\II_{Tr_Q(\X)}(d, d))&= & h^0(Q,
\II_{Tr_Q(\X)- \Lambda}(d-(2t+2),d)) \\
 &=& h^0(Q,\II_{Tr_Q(\X)-\Lambda}(t,3t+2)).
\end{eqnarray*}
  Hence we need to show that $h^0(Q,\II_{Tr_Q(\X)-\Lambda}(t,3t+2))=
  0,$ where
$Tr_Q(\X)- \Lambda= X_1+ P_1+\cdots+
 P_{t+3}\subset Q.$
 To see this,  we wish to construct a specialization of
 $Tr_Q(\X)-\Lambda$ with the desired vanishing,  we then must verify the residual and the
 trace in this new situation.

We start by choosing $t$ lines $M_1,\ldots,M_t$ of the first ruling
on $Q$, $M_i\neq M$. Next,
 let $Y$ be the scheme obtained from $Tr_Q(\X)- \Lambda$ by specializing
    the $t(3t+3)$  points of
$X_1$  onto the lines $M_i$ in such a way that each of these lines
contains exactly $3t+3$ of these points, and by  specializing the
remaining $2t$ points of $X_1$ onto the line $M$ (this is possible
because $t(3t+5)= 2t+ t(3t+3)$).

Now suppose that $C$ is a curve of $H^0(Q,\II_Y(t,3t+2))$, i.e. a
curve on $Q$ of type $(t,3t+2)$ containing $Y$.
 As we have just seen,   the line $M$ and also each  line $M_i$, $(1\leq i\leq t)$,  contains $3t+3$
points of $Y$. The fact that $C$ contains these points forces $C$ to
have the lines $M, M_i$ as fixed components (since otherwise $C$
must intersect $M$ (resp. $M_i$) at $3t+2$ points, while $C$ already
pass through the $3t+3$ points of $M$ (resp. $M_i$), which  is
impossible);
 but the number of these lines is $t+1$ and they are placed in the
 first ruling, which is a contradiction with the type $(t,3t+2)$ of $C$. So such a
$C$ cannot exist, i.e., we have proved that $h^0(Q,\II_Y(t,3t+2))=
0$. Then by semicontinuity one can deduce that
$h^0(Q,\II_{Tr_Q(\X)-\Lambda}(t,3t+2))= 0,$ which is equivalent to
$$h^0(Q,\II_{Tr_Q(\X)}(d, d))= 0.$$
 Finally,
from Remark \ref{rem} we get the conclusion.
\end{proof}
\section{Proof  in $\PP^4$}\label{s4}
In this section we will prove Theorem \ref{th2} for the case $n= 4$,
which for convenient we state again.

\vspace{0.2cm}
$\St^*(4,d)$:
  Let $d\geq 3$ and
$$r= \left \lfloor {{d+4 \choose 4} -  (4d+1)}\over {d+1}\right \rfloor
 ; \ \ \ \
q={d+4 \choose 4} - (4d+1) - r(d+1).
 $$
Let the scheme  $X\subset \PP^4$ be a generic union of $r$ lines
$L_1,\ldots,L_r$, one double line $2L$ and $q$  points
$P_1,\ldots,P_q$ lying on a generic line $M$.  Then $X$ has good
postulation, i.e.,
$$h^1(\II_X(d))= h^0(\II_X(d)) =   {d+4 \choose 4} -(4d+1) - r(d+1)- q= 0.$$
\begin{proof}
Let us begin with the case $d=3$. In this case we have $r= 5$, and
$q= 2$, so $X= 2L+ L_1+\cdots+ L_5+ P_1+ P_2\subset \PP^4$.

Pick a generic hyperplane $H\subset \PP^4$. Now specialize the lines
$L, L_1$ and also the points $P_1, P_2$ into $H$, and denote by $\X$
the specialized scheme.

 On the one hand we obtain $Res_H(\X)= L+L_2+\cdots+L_5\subset
 \PP^4,$
 that is, $Res_H(\X)$ is union of $5$ generic lines. Thus by
 Hartshorne--Hirschowitz theorem, Theorem \ref{HH th}, we immediately
 get
 $$h^0(\II_{Res_H(\X)}(2))= {{2+4}\choose{4}}- 15= 0.$$
 On the other hand we have $$Tr_H(\X)= 2L|_H+ L_1+ S_2+ \cdots+ S_5+
 P_1+ P_2\subset H,$$ where $L_i\cap H= S_i$, $(2\leq i\leq 5).$
 This means that $Tr_H(\X)$ is  generic union of one double line, one
 simple line, and $6$ simple points in $H\cong \PP^3$.
 As we observed in Section \ref{s3}, $\St^*(3,3)$ holds, which
 implies that $\St(3,3)$ holds.
 Now from   $\St(3,3)$, with $s= 1$,  we get that the scheme
   $2L|_H+ L_1\subset H\cong\PP^3$
 has good postulation in degree $3$, i.e.,
$$h^0(\II_{2L|_H+ L_1}(3))= {3+3\choose 3}- 10- 4= 6.$$ Since $P_1,
P_2$ and $S_i$,  $(2\leq i\leq 5),$  are $6$ generic points in $H$,
 we get
 $$h^0(H, \II_{Tr_H(\X)}(3))= 0.$$
Now by Remark \ref{rem} it follows that $h^0(\II_{X}(3))= 0.$

Let us consider the case $d= 4$. Then $r= 10$ and $q= 3.$
 We observe that $$X= 2L+ L_1+\cdots+ L_{10}+ P_1+ P_2+ P_3\subset
 \PP^4,$$
 where $P_1, P_2, P_3$ are generic points lying on the line $M$.

Fix a generic hyperplane $H\subset \PP^4$.  Let $\X$ be the scheme
obtained from $X$ by specializing the lines $L$ and $L_1, L_2, L_3$
into $H$.

We have
 $$Res_H(\X)= X_1+ P_1+ P_2+ P_3\subset \PP^4,$$
 where $X_1= L+ L_4+ \cdots+ L_{10}$.

 Applying Hartshorne--Hirschowitz theorem  to $X_1$, which is union of $8$ generic lines in $\PP^4$, yields
 $$h^0(\II_{X_1}(3))= {3+4\choose 4}- 32= 3;$$
 and also to $X_1+M$, which is union of $9$ generic lines in $\PP^4$, yields
 $$h^0(\II_{X_1+M}(3))= \max\left\{{3+4\choose 4}- 36, 0 \right\}= 0.$$
 Hence by Lemma \ref{line} we get
 $$h^0(\II_{Res_H(\X)}(3))= 0.$$

 Moreover, we have
 $$Tr_H(\X)= 2L|_H+ L_1+L_2+ L_3+ S_4+\cdots+ S_{10}\subset H,$$
 where $L_i\cap H= S_i$, $(4\leq i\leq 10)$.

By setting $X_2= 2L|_H+ L_1+L_2+ L_3$, we see that $X_2$ is generic
union in $H\cong \PP^3$ of one double line and $3$ simple lines, so
by $\St(3,4)$, with $s= 3$, we obtain
$$h^0(\II_{X_2}(4))= {4+3\choose 3}- 13- 15= 7.$$
Notice that the points $S_4,\ldots, S_{10}$ are $7$ generic points
in $H$, therefor
$$h^0(H, \II_{Tr_H(\X)}(4))= 0.$$
This together with $h^0(\II_{Res_H(\X)}(3))= 0$ implies that
$h^0(\II_{\X}(4))= 0$, and from here, by semicontinuity, it follows
the conclusion, which finishes the proof in this case.

Now assume  $d\geq 5$. The rest of the proof will be by induction on
$d$.

We start by letting
 $$r'= \left \lfloor
{{d+3 \choose 4} -  (4(d-1)+1)- q}\over d\right\rfloor;$$
 $$q'={d+3 \choose 4} - (4(d-1)+1) - r'd- q;$$
 $$x= r-r'-2q';$$
 further,  noting that $r', q', x\geq 0$
 (see
 Appendix, Lemma \ref{ap1}).

 Recall that the scheme $$X= 2L+L_1+\cdots+ L_r+ P_1+\cdots+ P_q\subset
 \PP^4,$$
 is  generic union of the double line $2L$, the $r$ simple lines $L_i$,
 and the $q$ points $P_i$ belonging to the generic line $M$.

 Fix a generic hyperplane $H\subset \PP^4$.
 In order to prove that
$X$ has good postulation in degree $d$, we construct a scheme  $\X$
obtained from $X$ by combining specializations and degenerations as
follows:
\begin{itemize}
\item  specialize the first $x$  lines $L_i$  into $H$, and call the resulting set of lines $X_1$;
\item
  degenerate the next $q'$ pairs of  lines $L_i$, so that they become
   $q'$ sundials $$\widehat{C_i}= C_i+ 2N_i|_{H_i};\ \ \ (1\leq i\leq q'),$$
   where  $C_i$ is a
   degenerate conic, $H_i\cong \PP^3$ is a generic linear space containing
   $C_i$   and $2N_i|_{H_i}$ is a double point in $H_i$ with support at the singular
   point of $C_i$, furtheremore,
specialize $\Ci$ in such a way that $C_i\subset H$, but
$2N_i|_{H_i}\not\subset H$,
   and call the
   resulting scheme of sundials  $X_2$, that is $$X_2= \widehat{C_1}+\cdots+ \widehat{C_{q'}},$$ with the
   property that the degenerate conics $C_i$ lie in $H$, but  $2N_i|_{H_i}\not\subset H$;
    \item leave the remaining simple lines $L_i$, which are   $r'= r-x- 2q'$ lines,
    generic not lying in $H$, and call this collection of lines
     $X_3$;
\end{itemize}
then let
 $$\X= 2L+ X_1+ X_2+ X_3+ P_1+\cdots+ P_q\subset \PP^4.$$

 We need to show that $h^0(\II_{\X}(d))= 0,$ which clearly implies that  $h^0(\II_{X}(d))= 0$.
  To do that, by Castelnuovo's inequality, it would be enough to show that
 $h^0(\II_{Res_H(\X)}(d-1))= 0$, and
 $h^0(\II_{Tr_H(\X)}(d))=0.$

First we verify the residual, which is
$$Res_H(\X)= 2L+ Res_H(X_2)+ X_3+ P_1+\cdots+ P_q\subset
\PP^4,$$ where $Res_H(X_2)= N_1+\cdots+ N_{q'}$.
 Recall that the points $P_i$ are $q$ generic lying on the line $M$.
 In order to apply Lemma \ref{line} to get  $h^0(\II_{Res_H(\X)}(d-1))= 0$,
  it suffices to prove  the two following
 equalities
  $$h^0(\II_{2L+ Res_H(X_2) + X_3}(d-1))= q;$$    $$h^0(\II_{2L+ Res_H(X_2) + X_3+ M}(d-1))=0.$$
 By the induction hypothesis we have that $\St^*(4,d-1)$ holds,
 then  $\St(4,d-1)$ holds. Now by applying $\St(4,d-1)$ to
 the scheme  $2L+ X_3$, which consists of one double line and $r'$
 generic lines, we get
 \begin{eqnarray*}
 h^0(\II_{2L+X_3}(d-1))&=&{d+3\choose 4}- (4(d-1)+1)- r'd \\
 &=& q+q'.
 \end{eqnarray*}
Since  $Res_H(X_2)$ consists of $q'$ generic points, it immediately
follows
\begin{equation} \label{1}
h^0(\II_{2L+X_3+ Res_H(X_2)}(d-1))= q.
\end{equation}
 In the same way, by applying
$\St(4,d-1)$ to
 the scheme  $2L+ X_3+ M$, which consists of one double line and $r'+1$
 generic lines, we get
\begin{eqnarray*}
  h^0(\II_{2L+X_3+M}(d-1))&=&\max\left\{{d+3\choose 4}- (4(d-1)+1)- (r'+1)d, 0\right\} \\
 &=& \max \{q+q'-d, 0\},
 \end{eqnarray*}
and therefor
\begin{equation}\label{2}
h^0(\II_{2L+X_3+ M+ Res_H(X_2)}(d-1))= \max\{q-d, 0\}= 0.
\end{equation}
Hence by (\ref{1}) and (\ref{2}) we get
$$h^0(\II_{Res_H(\X)}(d-1))= 0,$$  so we are done
with the residual scheme.

 Now  we treat the trace scheme $Tr_H(\X)$, which we denote  by $T$ for short,  that is
 $$T= Tr_H(\X)= 2R|_H+ X_1+ C_1+\cdots+ C_{q'}+ X'_3\subset H\cong
 \PP^3,$$
 where $L\cap H= R$  thus $2L\cap H= 2R|_H$ is a double point in $H$,
   and  $X'_3= Tr_H(X_3)$ is a generic collection of $r'$
 simple points; moreover, recall that $X_1$ is made by $x$ generic
 lines, where $x= r-r'-2q'$ as defined before.

 We must prove that $h^0(H,\II_{T}(d))= 0$.
 In order to do this,  we wish to construct a specialization of
 $T$,   with the desired vanishing, but this time our specialization will be via
 a smooth quadric surface.
 Since our investigations of  $T$ will be done  in   $H$, as the ambient
 space, so for simplicity of notation we will from now on
 write  $\PP^3$ instead of $H$, as well as,  $2R$ instead of
 $2R|_H$.

 Let $Q\cong \PP^1\times \PP^1$ be a smooth quadric in $\PP^3$.
 Notations and terminology concerning  $Q$
are those of the Section \ref{s3}.
 Let
 $$\hat{r}= \left \lfloor (d+1)^2- (d+2)q' -2x\over {d-1}\right
\rfloor,$$
 $$\hat{q}= (d+1)^2- (d+2)q' -(d-1)\hat{r}  -2x.$$
 Note that $\hat{r}\geq 0$, and so $\hat{q}\geq 0$ (see Appendix,   Lemma \ref{ap3} (i)).

To begin,  let $M_{1,i}, M_{2,i}$ be the two lines which form the
degenerate
 conic $C_i$, $(1\leq i\leq q')$, this means $C_i= M_{1,i}+ M_{2,i}$,
 and
  let $S_1,\ldots, S_{r'}$ be the points
 of $X'_3$.
 Because of the inequalities  $\hat{r}\leq x$   and $\hat{q}\leq
 r'$, (both are proved in Appendix, Lemma \ref{ap3} (ii), (iii)),
   we can specialize $T$
 in the following way:

 Let $\T$ be the scheme obtained from $T$ by specializing the
 degenerate conics $C_i$ and $\hat{r}$   lines $L_1,\ldots, L_{\hat{r}}$ of $X_1$
 in such a way that the lines
 $M_{1,1}\ldots,M_{1,q'}$  and $L_1, \ldots, L_{\hat{r}}$  become
  lines belonging to the first ruling
 of $Q$,  and by
 specializing  $\hat{q}$ points $S_1,\ldots, S_{\hat{q}}$ of
 $X'_3$ onto $Q$
  (the lines $M_{2,1}\ldots,M_{2,q'}$ and the other lines $L_{\hat{r}+1},\ldots, L_x$ of
  $X_1$, also the remaining points $S_{\hat{q}+1},\ldots, S_{r'}$ of $X'_3$ and the point $R$, remain
  generic not lying on $Q$).

 Next, we perform the process of treating the residual and the trace
 of the specialized scheme $\T$, with respect to $Q$, to get $h^0(\II_{\T}(d))= 0$.

We have
\begin{eqnarray*}
Res_Q(\T)&=& 2R+ M_{2,1}+\cdots+ M_{2,q'}+ L_{\hat{r}+1}+\cdots \\
 && +L_x+
 S_{\hat{q}+1}+\cdots+ S_{r'}\subset \PP^3.
 \end{eqnarray*}
Observe that the scheme  $Res_Q(\T)-( S_{\hat{q}+1}+\cdots+ S_{r'})$
is generic union of  one double point and $q'+x-\hat{r}$ lines in
$\PP^3$, and that $d-2\geq 3$, thus by Corollary \ref{CCG} we get
\begin{eqnarray*}
h^0(\II_{Res_Q(\T)- (S_{\hat{q}+1}+\cdots+ S_{r'})}(d-2))&=&
{d-2+3\choose 3}-4- (q'+x-\hat{r})(d-1)\\
&=& r'-\hat{q},
\end{eqnarray*}
(the last equality is proved in Appendix, Lemma \ref{ap3} (v)).
Moreover,  the points $S_{\hat{q}+1},\ldots,S_{r'}$ are $r'-
\hat{q}$ generic points, so we immediately get
$$ h^0(\II_{Res_Q(\T)}(d-2))= 0.$$

Now it remains to consider the trace scheme.
 We first notice that
 $M_{2,i}$ meets $Q$ in the two points which are $(M_{1,i}\cap
 M_{2,i})$   and another point, which we denote by $S'_i$, also
 recall  that $M_{1,i}\subset Q$,
  so we have that  $C_i\cap Q= M_{1,i}+ S'_i$, $(1\leq i\leq q')$.
Similarly, $L_{j}$, $(\hat{r}+1\leq j\leq x)$, meets $Q$ in two
points, then  $Tr_Q(L_{\hat{r}+1}+\cdots+ L_x)$ is a  collection of
$2(x-\hat{r})$ points, which we denote by $T_1$.
 Thus we obtain
 \begin{eqnarray*}
Tr_Q(\T)&= &L_1+ \cdots+L_{\hat{r}}+T_1+ M_{1,1}+ \cdots+ M_{1,q'}\\
&&+ S'_1+\cdots+ S'_{q'}+
 S_1+\cdots+ S_{\hat{q}} \subset Q.
 \end{eqnarray*}
  Since   the lines $M_{1,i}$   and $L_j$,  $(1\leq i\leq q';\  1\leq j\leq \hat{r})$,
   are contained  in the first ruling of
  $Q$, furthermore $d\geq q'+\hat{r}$ (see
Appendix, Lemma \ref{ap3} (iv)),
   then  all of these lines are fixed components for the curves of
$H^0(Q,\II_{Tr_Q(\T)}(d,d))$.
 Set $\Lambda= L_1+\cdots+L_{\hat{r}}+ M_{1,1}+\cdots+ M_{1,q'}\subset Q$.
  Now by removing the fixed
component $\Lambda$, and by using the fact that the points $S'_i$,
$S_k$, $(1\leq i\leq q'; \ 1\leq k\leq \hat{q})$, are generic on
$Q$,  as well as the points of
 $T_1$,
 we conclude that
\begin{eqnarray*}
h^0(Q,\II_{Tr_Q(\T)}(d,d))&=&h^0(Q,\II_{Tr_Q(\T)-\Lambda}(d-q'-\hat{r},d)) \\
&=& (d-q'-\hat{r}+1)(d+1)- (q'+\hat{q}+2x-2\hat{r})\\
&=& (d+1)^2-q'(d+2)- \hat{r}(d-1)-2x-\hat{q}\\
&=& 0.
\end{eqnarray*}
Putting together $ h^0(\II_{Res_Q(\T)}(d-2))= 0$ and
$h^0(Q,\II_{Tr_Q(\T)}(d,d))= 0$ we have $h^0(\II_{\T}(d))= 0$,
therefore, by semicontinuity, we have $h^0(\II_T(d))= 0$. This
completes the proof.
\end{proof}
\section{Proof  in $\PP^n$ for $n\geq 5$}\label{sn}
 We
come to the   general case $n\geq 5$. Now we have the bases for our
inductive approach, we are ready to prove Theorem \ref{th2} in the
general setting.

\vspace{0.2cm}
$\St^*(n,d)$:
 Let  $n,d\in\NN$, and $n\geq 5, d\geq 3$.
 Let
$$r= \left \lfloor {{d+n \choose n} -  (nd+1)}\over {d+1}\right \rfloor
 ; \ \ \ \
q={d+n \choose n} - (nd+1) - r(d+1).
 $$
Let the scheme  $X\subset \PP^n$ be a generic union of $r$ lines
$L_1,\ldots,L_r$, one double line $2L$ and $q$  points
$P_1,\ldots,P_q$ lying on a generic line $M$.  Then $X$ has good
postulation, i.e.,
$$h^1(\II_X(d))= h^0(\II_X(d)) =   {d+n \choose n} -(nd+1) - r(d+1)- q= 0.$$
\begin{proof}
We will prove the theorem by induction on $d$. We proceed to the
general case of $n\geq 5$, noting that $\St^*(3,d)$ and $\St^*(4,d)$
have been proved.

 To begin, let
  $$r'= \left \lfloor
{{d-1+n \choose n} -  (n(d-1)+1)- q}\over d\right\rfloor;$$
 $$q'={d-1+n \choose n} - (n(d-1)+1) - r'd- q;$$
 $$x= r-r'-2q';$$
 we can check that $r', q', x\geq 0$ (see Appendix, Lemma
 \ref{ap1}).

 Let $H\subset \PP^n$ be a generic hyperplane.
 For the purpose of getting $h^0(\II_X(d))= 0$,  we wish to find  a
 scheme $\X$ obtained from $X$ by combining specializations and
 degenerations so that
  the desired vanishing can be achieved. Now  we
  construct the required
 $\X$ in the following way, which is analogous to the one used in $\PP^4$ in the previous section:
\begin{itemize}
\item  specialize the first $x$  lines $L_i$  into $H$, and call the resulting set of lines $X_1$;
\item
  degenerate the next $q'$ pairs of  lines $L_i$, so that they become
   $q'$ sundials $$\widehat{C_i}= C_i+ 2N_i|_{H_i};\ \ \ (1\leq i\leq q'),$$
   where  $C_i$ is a
   degenerate conic, $H_i\cong \PP^3$ is a generic linear space containing
   $C_i$   and $2N_i|_{H_i}$ is a double point in $H_i$ with support at the singular
   point of $C_i$, furtheremore,
specialize $\Ci$ in such a way that $C_i\subset H$, but
$2N_i|_{H_i}\not\subset H$,
   and call the
   resulting scheme of sundials  $X_2$, that is $$X_2= \widehat{C_1}+\cdots+ \widehat{C_{q'}};$$
    \item leave the remaining simple lines $L_i$, which are   $r'= r-x- 2q'$ lines,
    generic not lying in $H$, and call this collection of lines
     $X_3$;
\end{itemize}
then let
 $$\X= 2L+ X_1+ X_2+ X_3+ P_1+\cdots+ P_q\subset \PP^n.$$

 To  show that $h^0(\II_{\X}(d))= 0,$  by Castelnuovo's inequality, our goal will be to show
 that the following vanishings
 $$h^0(\II_{Res_H(\X)}(d-1))= 0; \ \ \ h^0(H,\II_{Tr_H(\X)}(d))=0.$$

With regard to residual, we have
$$Res_H(\X)= 2L+ Res_H(X_2)+ X_3+ P_1+\cdots+ P_q\subset
\PP^n,$$ where $Res_H(X_2)= N_1+\cdots+ N_{q'}$.

By the induction hypothesis we know that $\St^*(n, d-1)$ holds,
which implies that $\St(n,d-1)$ holds (note that $n\geq 5$, then if
$d-1= 2$, by Proposition \ref{d2} we also  have that  $\St(n,2)$
holds). So we can apply $\St(n,d-1)$ to
 the scheme  $2L+ X_3$,   as well as, to the scheme $2L+X_3+M$,
 therefore
 \begin{eqnarray*}
 h^0(\II_{2L+X_3}(d-1))&=&{d-1+n\choose n}- (n(d-1)+1)- r'd \\
 &=& q+q';
 \end{eqnarray*}
\begin{eqnarray*}
  h^0(\II_{2L+X_3+M}(d-1))&=&\max\left\{{d-1+n\choose n}- (n(d-1)+1)- (r'+1)d, 0\right\} \\
 &=& \max \{q+q'-d, 0\}.
 \end{eqnarray*}
 Observe that  $Res_H(X_2)$ is made by $q'$ generic points, so we get
\begin{equation} \label{3}
h^0(\II_{2L+X_3+ Res_H(X_2)}(d-1))= q;
\end{equation}
\begin{equation}\label{4}
h^0(\II_{2L+X_3+ M+ Res_H(X_2)}(d-1))= \max\{q-d, 0\}= 0.
\end{equation}
 Having (\ref{3})and (\ref{4}), moreover,   recalling
that
  the points $P_i$ are $q$ generic points lying on the line $M$,
   we can now   apply Lemma  \ref{line},
 hence
$$h^0(\II_{Res_H(\X)}(d-1))= 0,$$
as we wanted.

Now, we consider  trace scheme $Tr_H(\X)$, which we denote by $T$
for short,
 $$T= Tr_H(\X)= 2R|_H+ X_1+ C_1+\cdots+ C_{q'}+ X'_3\subset H\cong
 \PP^{n-1},$$
 where $L\cap H= R$  thus $2L\cap H= 2R|_H$ is a double point in $H$,
   and  $X'_3= Tr_H(X_3)$ is a generic collection of $r'$
 simple points, which we denote by $S_1,\ldots, S_{r'}$.  In addition, recall that $X_1$ is made by $x$ generic
 lines, where $x= r-r'-2q'$ as defined before.

 For simplicity in the  notation,  we will
 henceforward
 write  $\PP^{n-1}$ instead of $H$, as well as,  $2R$ instead of
 $2R|_H$.

 In order to verify the scheme  $T$, we make a specialization $\T$ of $T$ via a fixed hyperplane as follows:
  we start by setting
 $$\bar{r}= \left \lfloor {{d+n-2 \choose n-2} -  (n-1)- r+r'}\over {d}\right
\rfloor;$$
 $$\bar{q}={d+n-2 \choose n-2} - (n-1)- \rr d- r+ r',$$
 also noting that $\bar{r}, \bar{q}\geq  0$ (Appendix, Lemma
 \ref{ap2} (i)).
Pick a generic hyperplane $H'$ in  $\PP^{n-1}$.
 Now using  the inequalities  $\bar{r}\leq x$   and $\bar{q}\leq
 r'$, (both are proved in Appendix, Lemma \ref{ap2} (ii), (iii)),
   we specialize the  lines $L_1,\ldots, L_{\bar{r}}$ of
   $X_1$, also  the  points $S_1,\ldots, S_{\bar{q}}$ of
   $X'_3$  and the point $R$ into $H'$, and we denote by $\T$ the
   specialized scheme (note that the other lines
   $L_{\bar{r}+1},\ldots, L_x$ of $X_1$, the degenerate conics
   $C_i$, and the other points of $X'_3$ remain generic outside
   $H'$).

Now in order to prove  that $h^0(\PP^{n-1},\II_{T}(d))= 0$, by
semicontinuity, our next goal will be to prove that
$h^0(\PP^{n-1},\II_{\T}(d))= 0$.

 $L_i$ meets $H'$ at one point, $(\bar{r}+1\leq i\leq x)$, so
 $Tr_{H'}(L_{\bar{r}+1}+\cdots+ L_x)$ is a union of $x-\bar{r}$
 points, which we denote by $T_1$. Moreover, $C_j$ meets $H'$ in two
 points, $(1\leq j\leq q')$,  then $Tr_{H'}(C_1+\cdots+ C_{q'})$ is a collection of $2q'$
 points, which we denote by $T_2$.  Accordingly with these notations,   we have
\begin{eqnarray*}
Tr_{H'}(\T)&=& 2R|_{H'}+ L_1+\cdots+ L_{\bar{r}}+ T_1+ T_2 \\
 &&+ S_1+\cdots+ S_{\bar{q}}\subset H'\cong \PP^{n-2}.
\end{eqnarray*}
First we apply Corollary \ref{CCG} to the scheme
$2R|_{H'}+L_1+\cdots+ L_{\bar{r}}$, which implies that
\begin{eqnarray*}
h^0(H', \II_{{2R|_{H'}+L_1+\cdots+ L_{\bar{r}}}}(d))&=&
{d+n-2\choose
n-2} -(n-1) -\bar{r}(d+1) \\
&=& \bar{q}-\bar{r}+r-r' \\
&=& \bar{q} -\bar{r} +2q'+x,
\end{eqnarray*}
next, by the fact that the schematic union $(T_1+ T_2+ S_1+\cdots+
S_{\bar{q}})$ is a generic union of $x-\bar{r}+ 2q'+ \bar{q}$ simple
points,  we immediately get
\begin{equation}\label{5}
h^0(H', \II_{Tr_{H'}(\T)}(d))= 0,
\end{equation}
so we are finished with the trace scheme.

Then we are left  with  the residual of $\T$ with respect to
$H'\cong \PP^{n-2}$, which is
\begin{eqnarray*}
Res_{H'}(\T)&=& R+ C_1+\cdots+ C_{q'}+ L_{\bar{r}+1}+ \cdots+ L_x
\\ && + S_{\bar{q}+1}+\cdots+ S_{r'}\subset \PP^{n-1}.
\end{eqnarray*}
 It is the existence of the degenerate conics $C_i$ that impedes us
 to directly investigate the residual scheme.
 Our method to afford this difficulty is to
 take a  degeneration of $Res_{H'}(\T)$, but  using a  different way to do so.
Indeed, according to the observation of \S \ref{dege} saying that a
sundial can be considered as a degeneration of a degenerate conic
together with a simple point,  we then  degenerate $q'$
 points   $S_{\bar{q}+1},\ldots, S_{\qq+q'}$ together with $q'$ conics
 $C_i$  so that they become $q'$ sundials $\Ci$  having singularity  at
 these points,
    (it is
 possible because  $q'\leq r'-\bar{q}$, Appendix, Lemma \ref{ap2}
 (iii)).  We set $\Gamma= R+ S_{\qq+q'+1}+\cdots+ S_{r'}.$

 Let $Y$ be the scheme obtained from $Res_{H'}(\T)$ by this
 degeneration, more precisely,
 $$Y= \widetilde{C_1}+\cdots+ \widetilde{C_{q'}}+ L_{\bar{r}+1}+ \cdots+
 L_x+ \Gamma\subset \PP^{n-1}.$$
 The scheme $Y-\Gamma$  is generic union of $q'$ sundials
and
 $x-\bar{r}$ lines in $\PP^{n-1}$,  so by  Theorem \ref{sundial} it has good postulation, in other
 words
 \begin{eqnarray*}
h^0(\PP^{n-1},\II_{Y-\Gamma}(d-1))&=&  {d-1+n-1\choose n-1}-
(2q'+x-\bar{r})d\\
&=& r'-q'-\bar{q}+1,
\end{eqnarray*}
the computations to get the last equality can be found in Appendix,
Lemma \ref{ap2} (iv). Since $\Gamma$ is  generic union of
$r'-\bar{q}+1-q'$ points, it then  immediately follows that
$$h^0(\PP^{n-1}, \II_Y(d-1))= 0,$$
and from here, again by semicontinuity, we obtain
$$h^0(\PP^{n-1}, Res_{H'}(\T)(d-1))= 0.$$
This together with (\ref{5}), by Castelnuovo's inequality,  yields
that
 $$h^0(\PP^{n-1}, \II_{\T}(d))= 0,$$
 and this is in fact what we wanted to show, hence the proof is complete.
\end{proof}
\section{On Conjecture \ref{con}}\label{scon}
 Now coming back to our Conjecture \ref{con},
 we will prove it  only in a special case.
\subsection{Some evidence for Conjecture \ref{con}}\label{evi}
The main result of this paper, Theorem \ref{th1},  attracts our
attention to a natural class of objects that is schemes  $X$ of
lines and one fat linear space in projective space. In fact, the
geometry of the exception that we determined in Theorem \ref{th1}
leads us to conjecture that it can be generalized somehow  to the
families  of lines and one fat linear space.
 The basic motivation lies in the
fact that, no defective cases with respect to the linear system
$|\II_X(d)|$ have been discovered, unless $d=m$, where $m$ is the
multiplicity of that linear space.
 So we hope the following conjecture,  which exactly describes the
failure of $X$ to have good postulation.
\begin{conj}[Conjecture \ref{con} of the Introduction] Let  $n,d,r\in\NN$,
and $n\geq r+2\geq 3$. The scheme $X\subset \PP^n$ consisting of
$s\geq 1$ generic lines and one $m$-multiple  linear space $m\Pi$,
$(m\geq 2)$,  with $\Pi\cong \PP^r\subset \PP^n$,  always has good
postulation, except for the cases
$$\left\{n=r+3, m=d, 2\leq s\leq d\right\}.$$
\end{conj}
This conjecture would be in perfect analogy with Theorem \ref{th1}.
Note that it is a hard problem to prove it in general case, and
doing so requires the most sophisticated
 investigations with a lot of  technical details, in the setting
 of specialization and degeneration.

  Now  we  show that  the conjecture is true for the special case of $d=m$,
 which is in the center of our attention.
 Before proceeding to state and prove it,
let us introduce the following integer $\al_{(n,d;r,m)}$, for all
integers $n,r,d,m$ with $n> r$ and $d\geq m-1$, which we will use
throughout this section:
$$\al_{(n,d;r,m)}= \sum_{i=0}^{m-1}{r+d-i \choose r}{n+i-r-1 \choose
i}.$$
 Observe that $\al_{(n,d;r,m)}$ is exactly the Hilbert polynomial of
 $m\Pi$ in degree $d$, Lemma \ref{HP}.
 Moreover, when $d=m$ with a straightforward computation, one easily
 sees that:
$$\al_{(n,m; r,m)}= {n+m\choose n}-{n+m-r-1\choose n-r-1}.$$
\begin{prop}
 The scheme  $X\subset \PP^n, n\geq r+2\geq 3,$
consisting of $s\geq 1$ generic lines  and one $m$-multiple linear
space $m\Pi$, $(m\geq 2)$, with $\Pi\cong \PP^r\subset \PP^n$,  has
good postulation in degree $m$, i.e.,
\begin{eqnarray*}
h^0(\II_X(m))&=&\max \left \{ {n+m \choose n}  -\al_{(n,m; r,m)}-
s(d+1), 0 \right
\} \\
&=&\max \left\{ {n+m-r-1\choose n-r-1} -s(d+1), 0 \right\},
\end{eqnarray*}
except for  $\left\{n=r+3, 2\leq s\leq d\right\},$ in which case
the defect is ${s\choose 2}$.
\end{prop}
\begin{proof}
First notice that, the sections of $\II_X(m)$ correspond to degree
$m$ hypersurfaces in $\PP^n$ which, in order to contain $m\Pi$, have
to be cones whose vertex contains the linear space $\Pi$.

For $n=r+2$, it is easy to see that the linear system $|\II_X(m)|$
is empty, i.e.  $h^0(\II_X(m))=0$, that is what was expected.

For $n\geq r+3$, let us  consider the projection $X'$ of $X$ from
$\Pi$ into a generic linear subspace $\PP^{n-r-1}\subset \PP^n$.
Then we have that  the scheme $X'$ consists of  $s$ generic lines in
$\PP^{n-r-1}$, also that the following equality
 $$h^0(\PP^n,\II_X(m))= h^0(\PP^{n-r-1},\II_{X'}(m)).$$
In case $n>r+3$, we have  $n-r-1\geq 3$, therefor from
Hartshorne--Hirschowitz theorem \ref{HH th} it follows that
$$h^0(\PP^{n-r-1},\II_{X'}(m))= \max \left\{{m+n-r-1 \choose n-r-1} -s(d+1),
0\right\},$$ which is the expected value for  $h^0(\PP^n,
\II_X(m))$, so we are done in this case.

In case $n=r+3$, $X'$ is a generic union of $s$ lines in $\PP^2$.
  Hence, if $s>m$ we obviously have  $h^0(\II_{X'}(m))=
  0$, as expected.
 If $s\leq m$ we have $h^0(\II_{X'}(m))= {m-s+2 \choose 2},$
 on the other hand the expected value for
 $h^0(\II_{X}(m))$ is  $$\max\left\{{m+2\choose 2}- s(m+1), 0\right\},$$
 which we denote by $\exp h^0(\II_X(m))$.
 Thus for $s=1$,  we get that $h^0(\II_{X}(m))= {m+1\choose 2}$,  as expected; but  for
 $2\leq s\leq m$, we get that $h^0(\II_{X}(m))\neq \exp h^0(\II_X(m))$ and the defect is
$$h^0(\II_X(m))- \exp h^0(\II_X(m))= {s\choose 2},$$
which finishes the proof.
\end{proof}
\subsection{Final remark}
 A complete proof for  Conjecture \ref{con}  will be a substantial
 effort, however, we  believe that a method analogous to that
  presented  in \S \ref{st},  can be successfully applied for studying
  postulation problem for a generic scheme of lines and one fat linear space in $\PP^n$, and we plan
 to study this problem in the future.
 Indeed, if one can  provide  a proof for Conjecture \ref{con}
  for a generic union of lines and one fat line  in $\PP^n$, then  even interestingly enough, one may hope to generalize this
 approach to the  cases of lines and one fat linear space,  that seems to
 be quite
 difficult.
Actually,  compared with the proof we gave in this paper, in the
case of lines and one fat linear space we are forced
 to divide the proof in much more steps.
 While an  argument analogous  to Theorem \ref{th1}  works in a more
 complicated way for the higher dimensional ambient projective spaces, investigations in two initial ambient
 spaces cause troubles, and this is why we leave it for the future.
\section{Appendix: Calculations}\label{app}
\begin{lem}\label{ap1}
Let $n\geq 5, d\geq 3$ or  $n= 4, d\geq 5$.
 Let
$$r= \left \lfloor {{d+n \choose n} -  (nd+1)}\over {d+1}\right
\rfloor,\ \ q={d+n \choose n} - (nd+1) - r(d+1);$$
 $$r'= \left \lfloor
{{d-1+n \choose n} -  (n(d-1)+1)- q}\over d\right\rfloor,$$
 $$q'={d-1+n \choose n} - (n(d-1)+1) - r'd- q.$$
 Then
 \begin{itemize}
\item[(i)] $r'\geq 0$;
\item[(ii)] $r-r'-2q'\geq 0$.
\end{itemize}
\end{lem}
\begin{proof}
(i) Since $q\leq d$, we have
$${d-1+n \choose n} -  (n(d-1)+1)- q\geq {d-1+n \choose n}-
(n(d-1)+1)- d,$$
 so in order to show that  $r'\geq 0$ it is enough to show that
 \begin{equation}\label{6}
{d-1+n\choose n}- (n(d-1)+1)\geq d.
 \end{equation}
 First consider the case $n=4$ and $d\geq 5$, then  we obviously have
 $${d-1+n\choose n}- (n(d-1)+1)- d= {d+3\choose 4} -5d+3\geq 0.$$

 Now consider  the case $n\geq 5$ and $d\geq 3$. Notice that the  function
 ${d-1+n \choose n}- (n(d-1)+1)$  is an increasing function in
$n$, hence to get the conclusion it suffices to prove the inequality
(\ref{6})
 only for $n=5$.
 Now by letting $n=5$, we easily see that
\begin{eqnarray*}
{d-1+n \choose n} -  (n(d-1)+1)- d&=&{d-1+5 \choose 5}- (5(d-1)+1) -d\\
&= & {d+4 \choose 5} -6d+4\geq 0,
\end{eqnarray*}
the last inequality is surely holds for $d\geq 3$.

\vspace{0.2cm}
(ii) We have to prove that
$$ \left \lfloor {{d+n \choose n} -  (nd+1)}\over {d+1}\right
\rfloor\geq r'+2q'.$$
 Since $r'$ and $q'$ are integers, the inequality above is
 equivalent to the following
 $$ {{d+n \choose n} -  (nd+1)\over {d+1}}\geq r'+2q',$$
 hence, it is enough to prove that
 $${d+n \choose n} - (nd+1) -(d+1)r'- 2(d+1)q'\geq 0.$$
 We have
 \begin{eqnarray*}
& &{d+n \choose n} - (nd+1) -(d+1)r'- 2(d+1)q' \\
&=& {d+n \choose n}-2(d+1){d-1+n\choose n}+2(d+1)(n(d-1)+1)-(nd+1) \\
& &+2(d+1)q +(d+1)(2d-1)r'\\
&\geq & {d+n \choose n}-2(d+1){d-1+n\choose n}+2(d+1)(n(d-1)+1)-(nd+1) \\
& &+2(d+1)q +(d+1)(2d-1)\left\{{{{d-1+n\choose n}-
(n(d-1)+1)-q}\over {d}}-1\right\}\\
&=&{\frac{1}{d}}\left\{d{d+n\choose n}-(d+1){d-1+n\choose
n}-(n-1)+q(d+1)-d(d+1)(2d-1)\right\}\\
&=& {\frac{1}{d}}\left\{(n-1){d-1+n\choose n}-(n-1)+q(d+1)-d(d+1)(2d-1)\right\}\\
&\geq &{\frac{1}{d}}\left\{(n-1){d-1+n\choose n}-{d-1+n\choose
n}-d(d+1)(2d-1)\right\}\\
&=& {\frac{1}{d}}\left\{(n-2){d-1+n\choose n}-d(d+1)(2d-1)\right\}.
\end{eqnarray*}
For $n\geq 5$, we get
\begin{eqnarray*}
 &&(n-2){d-1+n\choose n}-d(d+1)(2d-1)\\
 && \geq  3{d+4 \choose 5}-
 d(d+1)(2d-1)\\
 &&= \frac{1}{40}d(d+1)\left\{(d+2)(d+3)(d+4)-80d+40\right\}\geq 0,
 \end{eqnarray*}
 it is quite immediate to check that the last inequality holds for
 all $d\geq 3$, hence we are done in the case $n\geq 5$.

For $n=4$, we get
\begin{eqnarray*}
 &&(n-2){d-1+n\choose n}-d(d+1)(2d-1)\\
 &&= \frac{1}{12}d(d+1)\left\{(d+2)(d+3)-24d+12\right\}\\
 &&=
 \frac{1}{12}d(d+1)(d-1)(d-18),
 \end{eqnarray*}
 so this is positive for all $d\geq 18$. This is what we wanted to show for
 $n=4$ and $d\geq 18$,   then we are left with
 $5\leq d\leq 17$.
 Now by a direct computation we get the desired inequality
 $r-r'-2q'\geq0$ in the case $n=4$ with $5\leq d\leq 17$ as follows:

$$\begin{tabular}{c|c|c|c|c|c|c|c|c|c|c|c|c|c}
  $d$ & 5&6&7&8&9&10&11&12&13&14&15&16&17\\
  \hline
  $r-r'-2q'$&7&9&2&10&9&10&17&9&30&34&17&35&32\\
\end{tabular}$$

\end{proof}
\begin{lem}\label{ap2}
Let $n\geq 5$ and $d\geq 3$. With the  notations as in Lemma
\ref{ap1}, let
$$\bar{r}= \left \lfloor {{d+n-2 \choose n-2} -  (n-1)- r+r'}\over {d}\right
\rfloor,$$
 $$\bar{q}={d+n-2 \choose n-2} - (n-1)- \rr d - r+ r'.$$ Then
 \begin{itemize}
\item[(i)] $\rr\geq 0$;
\item[(ii)] $\rr\leq r-r'-2q'$;
\item[(iii)] $r'\geq q'+\qq$;
\item[(iv)] ${d+n-2\choose n-1}-
(r-r'-\bar{r})d= r'-q'-\bar{q}+1.$
\end{itemize}
\end{lem}
\begin{proof}
(i) We will verify that
$${d+n-2 \choose n-2} -  (n-1)- r+r'\geq 0.$$
Since
$$r\leq {{{d+n \choose n} - (nd+1)}\over {d+1}}; \ \
 r'\geq {{{d-1+n \choose n} - (n(d-1)+1)-q}\over {d}}-1,$$
 we have
\begin{eqnarray*}
r-r'&\leq & \frac{1}{d+1}\left\{{d+n\choose n}-(nd+1)\right\}\\
&& -\frac{1}{d}\left\{{d-1+n\choose n}-(n(d-1)+1)-q\right\}+1\\
&=&  \frac{1}{d(d+1)}\left\{d{d+n\choose n}-(d+1){d-1+n\choose
n}-(n-1)+q(d+1)\right\}+ 1 \\
&=&  \frac{1}{d(d+1)}\left\{(n-1){d-1+n\choose
n}-(n-1)+q(d+1)+d(d+1)\right\}.  \\
\end{eqnarray*}
Then  we get
\begin{eqnarray*}
&& {d+n-2 \choose n-2} -  (n-1)- r+r'\\
&\geq & {d+n-2 \choose n-2} -  (n-1)\\
&& -  \frac{1}{d(d+1)}\left\{(n-1){d-1+n\choose
n}-(n-1)+2d(d+1)\right\}\\
&=& \frac{\mathbf{A}}{d(d+1)},
\end{eqnarray*}
where
$$\mathbf{A}= d(d+1){d+n-2\choose n-2}-(n-1){d-1+n\choose n}
-(n-1)(d^2+d-1)-2d(d+1).$$

A straightforward computation, yields
$$\mathbf{A}={1 \over n} {d+n-2\choose n-2}(nd^2-d^2+d)-(n-1)(d^2+d-1)-2d(d+1). $$
 Since  $n\geq 5, d\geq 3$ we have  $n(n-1)\leq {d+n-2\choose n-2}$,
 and from here it follows
\begin{eqnarray*}
\mathbf{A}&\geq & (n-1)(nd^2-d^2+d-(d^2+d-1))-2d(d+1)\\
&=& d^2(n^2-3n)-2d+n-1\geq 0,
\end{eqnarray*}
 which completes the proof.

\vspace{0.2cm}
(ii) In order to prove that $\rr\leq r-r'-2q'$, it suffices to prove
that
$$ {{{d+n-2 \choose n-2} -  (n-1)- r+r'}\over {d}}\leq r-r'-2q'+1,$$
which is equivalent to the following
$$r(d+1)-r'(d+1)-{d+n-2\choose n-2}+(n-1)-2q'd+d\geq 0.$$
From the definitions of   $r$ and $r'$, moreover   the inequality
$q'\leq d-1$, we get
\begin{eqnarray*}
&& r(d+1)-r'(d+1)-{d+n-2\choose n-2}+(n-1)-2q'd+d \\
&\geq &  {d+n\choose n}-(nd+1) -(d+1)
-\frac{d+1}{d}\left\{{d-1+n\choose
n}-(n(d-1)+1)-q\right\}\\
& & -{d+n-2\choose n-2}+(n-1)-2d^2+3d,
\end{eqnarray*}
which, by an easy computation, is equal to
$$\frac{1}{d}\left\{(n-1){n+d-2\choose n}
+n(d-1)-2d^3+2d^2-2d+1+q(d+1)\right\}.$$
 Now  we observe that
\begin{eqnarray*}
&&(n-1){n+d-2\choose n} +n(d-1)-2d^3+2d^2-2d+1+q(d+1) \\
&\geq & (n-1){n+d-2\choose n} +n(d-1)-2d^3+2d^2-2d+1,
\end{eqnarray*}
hence,  we will be done if we prove that
\begin{equation}\label{8}
(n-1){n+d-2\choose n} +n(d-1)-2d^3+2d^2-2d+1\geq 0.
\end{equation}

 For $n\geq 6$,  we have
$$(n-1){n+d-2\choose n} +n(d-1)-2d^3+2d^2-2d+1$$
$$\geq 5{d+4\choose 6}-(2d^3-2d^2-4d+5),$$
which, for  $d\geq 3$, is positive, as we wanted.

For $n=5$, the inequality (\ref{8}) becomes:
$$ 4{d+3\choose 5}-(2d^3-2d^2-3d+4)\geq 0,$$
which is true for $d\geq 5$, so we are left with $d=3, 4$ in the
case of $n=5$. But  direct computations show that also these cases
satisfy the required inequality $\rr\leq r-r'-2q'$. More precisely,
if $d=3$, we have $\rr= 3$ and $r-r'-2q'=5$; if $d=4$, we have $\rr=
5$ and $r-r'-2q'= 11$.

\vspace{0.2cm}
(iii) We want to prove that $q'+\qq\leq r'$. By the inequalities
$q', \qq\leq d-1$, which implies  $q'+\qq\leq 2d-2$, and also by the
following one
 $$r'\geq {{{d-1+n
\choose n} - (n(d-1)+1)-q}\over {d}}-1,$$ it is enough to prove that
$${{{d-1+n \choose n} - (n(d-1)+1)-q}\over {d}}-1\geq 2d-2,$$
i.e.
$${d-1+n \choose n} - (n(d-1)+1)-q-2d^2+d\geq 0.$$
Using $q\leq d$, we have to show that
$${d-1+n \choose n} - (n(d-1)+1)-2d^2\geq 0,$$
or, equivalently,
\begin{equation}\label{9}
  {d-1+n \choose n}
- n(d-1)\geq 2d^2+1.
\end{equation}
 Notice that the  function
 ${d-1+n \choose n} -n(d-1)$  is an increasing function in
$n$. For $n=5$, the inequality  (\ref{9}) becomes
$$ {d+4\choose 5}\geq 2d^2+5d-4,$$
 which holds for  $d\geq 4$. So it remains to check $q'+\qq\leq r'$
 in the case of  $d=3$ with $n=5$. In this case we can directly compute that
 $r'= 3$, $q'=1$, $\qq= 0$. Hence the case $n=5$ is done.

 For $n=6$, the inequality  (\ref{9}) becomes
$$ {d+5\choose 6}\geq 2d^2+6d-5,$$
which holds for $d\geq 4$. So we are left with $d=3$. A direct
computation in the case of $d=3$ with $n=6$ yields that  $r'=4$,
$q'=2$, $\qq= 0$. So we are done for $n=6$.

Finally, for $n=7$, the inequality  (\ref{9}) becomes
$$ {d+6\choose 7}\geq 2d^2+7d-6,$$
which is true for any $d\geq 3$.
 Now,  since ${d-1+n \choose n} -n(d-1)$
is an increasing function in $n$, we have proved (\ref{9}) for all
$n\geq 7$ and $d\geq 3$. That finishes the proof of part (iii).

\vspace{0.2cm}
(iv) We must check that
 $${d+n-2\choose n-1}-
(r-r'-\bar{r})d= r'-q'-\bar{q}+1,$$
 that is
 \begin{equation}\label{10}
{d+n-2\choose n-1}+(r'd+q')+(\rr d+\qq-r')= rd+1.
\end{equation}
From the definitions of $q', \qq$  we have
 $$r'd+q'= {d-1+n\choose n}-n(d-1)-1-q;$$
 $$\rr d+\qq -r'= {d+n-2\choose n-2}-(n-1)-r.$$
Now using these equalities and an easy computation yields
$${d+n-2\choose n-1}+(r'd+q')+(\rr d+\qq-r')$$
$$={d+n\choose n}-nd-r-q,$$
which by $q= {d+n\choose n}-(nd+1)-r(d+1)$ is equal to $(rd+1)$,
that is what we wanted (\ref{10}).
\end{proof}
\begin{lem}\label{ap3}
 Let $d\geq 5$. Let $r, r', q, q'$ be as in Lemma \ref{ap1} in the
 case $n=4$.
 Let
$$\hat{r}= \left \lfloor (d+1)^2- (d+2)q' -2(r-r'-2q')\over {d-1}\right
\rfloor,$$
 $$\hat{q}= (d+1)^2- (d+2)q' -(d-1)\hat{r} -2(r-r'-2q').$$ Then
\begin{itemize}
\item[(i)] $\hat{r}\geq 0$;
\item[(ii)] $\hat{r}\leq r-r'-2q'$;
\item[(iii)] $\hat{q}\leq r'$;
\item[(iv)] $q'+ \hat{r}\leq d;$
\item[(v)] $r'-\hat{q}= {d+1\choose
3}-4- (d-1)(r-r'-\hat{r}-q').$
\end{itemize}
\end{lem}
\begin{proof}
(i) We  need to show that
$$(d+1)^2- (d+2)q' -2(r-r'-2q')\geq 0,$$
that is
$$(d+1)^2-(d-2)q'-2(r-r')\geq 0.$$
 Recall:
$$r= \left \lfloor {d+4\choose 4}-4d-1\over {d+1}\right
\rfloor, \ \ \ q={d+4\choose 4}-4d-1-r(d+1);$$
$$r'= \left \lfloor {{d+3\choose 4}-4d+3-q}\over {d}\right\rfloor, \ \ \ q'= {d+3\choose 4}-4d+3-q-r'd.$$

 Let us start by computing
$(d-2)q'+2(r-r')$:
\begin{eqnarray*}
(d-2)q'+2r-2r' &=& (d-2){d+3\choose 4}
 -(d-2)(4d-3)\\
 && -(d-2)q-(d^2-2d+2)r'+2r \\
&\leq & \frac{\mathbf{A}}{d(d+1)},
\end{eqnarray*}
 where
 \begin{eqnarray*}
 \mathbf{A}&=& (d^2+d)(d-2){d+3\choose
 4}-(d^2+d)(d-2)(4d-3) \\
 && -(d^2+d)(d-2)q
 -(d+1)(d^2-2d+2)\left\{{d+3\choose 4}-4d+3-q\right\}\\
 && +2d\left\{{d+4\choose 4}-4d-1-(d+1)\right\} \\
 &=& 2d{d+4\choose 4}-2(d+1){d+3\choose 4}-2(d^2+d+3)+2q(d+1)\\
 &=& 6{d+3\choose 4}-2(d^2+d+3)+ 2q(d+1).
\end{eqnarray*}
Therefore we get
\begin{eqnarray*}
(d-2)q'+2(r-r')&\leq &\frac{\mathbf{A}}{d(d+1)}\\
&= & \frac{(d+2)(d+3)}{4}-\frac{2(d^2+d+3)}{d^2+d}+\frac{2q}{d},
\end{eqnarray*}
by noting that $\frac{2(d^2+d+3)}{d^2+d}\geq 2$ and that $q\leq d$,
it immediately follows
$$(d-2)q'+2(r-r')\leq \frac{(d+2)(d+3)}{4}.$$
Now from here we have
\begin{eqnarray*}
(d+1)^2-(d-2)q'-2(r-r')&\geq &(d+1)^2- \frac{(d+2)(d+3)}{4}\\
&= & \frac{3d^2+3d-2}{4}\geq 0,
\end{eqnarray*}
and this finishes the proof.

\vspace{0.2cm}
(ii)  In order to check that $\hat{r}\leq r-r'-2q'$, it suffices to
check that
$$ { (d+1)^2- (d+2)q' -2(r-r'-2q')\over {d-1}}\leq r-r'-2q'+1,$$
that is
$$(d+1)^2-(d+2)q'-2(r-r'-2q')\leq (d-1)(r-r'-2q')+(d+1),$$
or, equivalently
$$(d+1)(r-r')-dq'-(d+1)^2+(d-1)\geq 0.$$
Again, using the definitions of   $r$ and $r'$, moreover   the
inequality $q'\leq d-1$,  one gets
\begin{eqnarray*}
&& (d+1)(r-r')-dq'-(d+1)^2+(d-1) \\
&\geq& (d+1)r-(d+1)r'-2(d^2+1)\\
&\geq & {d+4\choose 4}-(4d+1)-(d+1) \\
&& -\frac{d+1}{d}\left\{{d+3\choose 4}-4d+3-q\right\}-2(d^2+1),
\end{eqnarray*}
which, by a short computation similar to that in part (i), is equal
to
$$\frac{1}{d}\left\{3{d+3\choose 4}-(2d^3+d^2+3d+3)+q(d+1)\right\}.$$
Now we have
\begin{eqnarray*}
&& 3{d+3\choose 4}-(2d^3+d^2+3d+3)+q(d+1) \\
&\geq &3{d+3\choose 4}-(2d^3+d^2+3d+3)\\
&= &\frac{1}{8}(d^4-10d^3+3d^2-18d-24),
\end{eqnarray*}
which, in fact for $d\geq 10$ is positive, as required.   Then  it
remains to check that the cases $5\leq d\leq 9$ satisfy $\hat{r}\leq
r-r'-2q'$. Computing each of these cases, we get the conclusion:

$$\begin{tabular}{c|c|c}
  $d$&$\hat{r}$&$r-r'-2q'$\\
  \hline
  5&5&7\\
  \hline
  6&6&9\\
  \hline
  7&2&2\\
  \hline
  8&5&10\\
  \hline
  9&4&9
\end{tabular}$$

\vspace{0.2cm}
(iii) To prove  $\hat{q}\leq r'$, by noting that $\hat{q}\leq d-2$,
it is enough to prove
$$\frac{{d+3\choose 4}-4d+3-q}{d}\geq d-1,$$
i.e.
\begin{equation}\label{13}
{d+3\choose 4}-4d+3-q-d(d-1)\geq 0.
\end{equation}
Observe that
\begin{eqnarray*}
&& {d+3\choose 4}-4d+3-q-d(d-1)\\
&\geq & {d+3\choose 4}-4d+3-d-d(d-1)\\
&=& {d+3\choose 4}-(d^2+4d-3).
\end{eqnarray*}
For $d\geq 5$, it is immediate to see that
$${d+3\choose 4}-(d^2+4d-3)\geq 0,$$
which gives (\ref{13}).

\vspace{0.2cm}
(iv) We will show that $d-q'-\hat{r}\geq 0.$
 We have
\begin{eqnarray*}
d-q'-\hat{r}&\geq &
d-q'-\frac{1}{d-1}((d+1)^2-(d-2)q'-2r+2r')\\
&=& \frac{1}{d-1}(2r-2r'-q'-3d-1),
\end{eqnarray*}
moreover,
 \begin{eqnarray*}
 && 2r-2r'-q'-3d-1\\
 &= &2r+(d-2)r'-{d+3\choose 4}+d-4+q\\
 &\geq & \frac{\mathbf{A}}{d(d+1)},
\end{eqnarray*}
where,
\begin{eqnarray*}
 \mathbf{A}&=& 2d{d+4\choose 4}-2d(4d+1)-2d(d+1)+(d-2)(d+1){d+3\choose
 4}\\
 && -(d-2)(d+1)(4d-3)-(d-2)(d+1)q-d(d+1)(d-2)\\
&& -(d^2+d){d+3\choose 4}
  +(d^2+d)(d-4)+(d^2+d)q.\\
\end{eqnarray*}
 After simple computations,
 one can easily find that
\begin{eqnarray*}
\mathbf{A}&= & 2d{d+4\choose 4}-2(d+1){d+3\choose 4}\\
&& -(4d^3+5d^2+d+6)+2(d+1)q\\
&= & 6{d+3\choose 4}-(4d^3+5d^2+d+6)+2(d+1)q\\
&\geq & 6{d+3\choose 4}-(4d^3+5d^2+d+6),
\end{eqnarray*}
which is positive for $d\geq 11$, hence we are left with $5\leq
d\leq 10$. Now by  direct calculations we get  $q'+\hat{r}\leq d$ in
these cases as follows:

$$\begin{tabular}{c|c|c|c}
  $d$&$q'$&$\hat{r}$&$q'+\hat{r}$\\
  \hline
  5&0&5&5\\
  \hline
  6&0&6&6\\
  \hline
  7&5&2&7\\
  \hline
  8&2&5&7\\
  \hline
  9&4&4&8\\
  \hline
  10&5&4&9
\end{tabular}$$

\vspace{0.2cm}
(v) We have to verify that
 $$r'-\hat{q}= {d+1\choose
3}-4- (d-1)(r-r'-\hat{r}-q'),$$
 that is
 \begin{equation}\label{14}
 (d-1)\hat{r}+ \hat{q}+ (d-2)r'+(d-1)q'= (d-1)r-{d+1\choose
 3}+4.
 \end{equation}
Rewrite the left hand side  as
$$((d-1)\hat{r}+\hat{q}+(d-2)q'-2r')+ (dr'+q').$$
Recalling that
$$(d-1)\hat{r}+\hat{q}+(d-2)q'-2r'= (d+1)^2-2r;$$
$$dr'+q'= {d+3\choose 4}-4d+3-q,$$
the left hand side of (\ref{14}) becomes:
$$ (d+1)^2-2r+ {d+3\choose 4}-4d+3-q$$
$$=(d+1)^2-2r+ {d+3\choose 4}-4d+3-{d+4\choose 4}+4d+1+(d+1)r$$
$$=(d+1)^2-{d+3\choose 3}+4+(d-1)r$$
$$= -{d+1\choose
 3}+4+ (d-1)r,$$
and we are done.
\end{proof}

\subsubsection*{Acknowledgment}
 I would like to thank
Professor M. V. Catalisano, for
 sharing  with me  many geometrical
insight about  techniques involved in the postulation problem during
my stay at the university of Genova, for suggesting that I study the
problem considered here, and particularly for her willingness to
read patiently an early version of this paper.



\begin{thebibliography}{CGG5}

\bibitem[AB14]{AB}
T. Aladpoosh, E. Ballico, {\em Postulation of disjoint unions of
lines and a multiple point}, Rend. Sem. Mat. Univ. Politec. Torino.
 72 (3--4) (2014), 127--145.

\bibitem[AH95]{AH}
J. Alexander, A. Hirschowitz, {\em Polynomial interpolation in
several variables}, J.  Alg. Geom. 4 (2) (1995), 201--222.

\bibitem[Bal10]{Balm}
E. Ballico, {\em On the Hilbert functions of disjoint unions of a
linear space and many lines in $\PP^n$}, International Mathematical
Forum, 5 (16) (2010), 787--798.

\bibitem[Bal11]{Bal}
E. Ballico, {\em Postulation of disjoint unions of lines and a few
planes}, J. Pure Appl. Algebra, 215 (4) (2011), 597--608.

\bibitem[Bal15]{B1}
E. Ballico, {\em Postulation of disjoint unions of lines and a
multiple point II}, Mediterr. J. Mat. (2015), 1--15.

\bibitem[Bal16]{B3}
E. Ballico, {\em On the maximal rank of a general union of a
multiple linear space and a generic rational curve}, Bol. Soc. Mat.
Mex. 22 (1) (2016), 13--31.

\bibitem[CCG10]{CCG1}
E. Carlini, M. V. Catalisano, A. V. Geramita, {\em Bipolynomial
Hilbert functions},  J. Algebra. 324 (4) (2010), 758--781.

\bibitem[CCG11]{CCG3}
E. Carlini, M. V. Catalisano, A. V. Geramita, {\em 3-dimensional
sundials}, Cent. Eur. J. Math. 9 (5) (2011), 949--971.

\bibitem[CCG12]{CCG2}
E. Carlini, M. V. Catalisano, A. V. Geramita, {\em Subspace
arrangements, configurations of linear spaces and the quadrics
containing them}, J. Algebra. 362 (3) (2012), 70--83.

\bibitem[CCG16]{CCG4}
E. Carlini, M. V. Catalisano, A. V. Geramita, {\em On the Hilbert
function of lines union one non-reduced point}, Ann. Sc. Norm.
Super. Pisa Cl. Sci (5) XV (2016), 69--84.

\bibitem[Cil00]{cilib}
C. Ciliberto,  {\em Geometric aspects of polynomial interpolation in
more variables and of Waring's problem}, in: Proceedings of the
European Congress of Mathematics, Barcelona, (2000), in:  Progress
in Math, Brikha\"user (2001), 289--316.

\bibitem[CH14]{CH}
S. Cooper, B. Harbourne, {\em Regina lectures on fat points}, in:
Connections between Algebra, Combinatorics and Geometry, vol. 76,
Springer Proceedings in Mathematics and Statistics,  Springer, New
York, (2014), 147--187.

\bibitem[Der07]{D}
H. Derksen, {\em Hilbert series of subspace arrangements}, J. Pure
Appl. Algebra. 209 (1) (2007), 91--98.

\bibitem[DS02]{Sid}
H. Derksen, J. Sidman, {\em A sharp bound for the
Castelnuovo--Mumford regularity of subspace arrangements}, Adv.
Math. 172 (2) (2002), 151--157.

\bibitem[DHST14]{DHST}
M. Dumnicki, B. Harbourne, T. Szemberg, H. Tutaj-Gasin'ska, {\em
Linear subspaces, symbolic powers and Nagata type conjectures}, Adv.
Math. 252 (2014), 471--491.

\bibitem[FHL15]{FHL}
G. Fatabbi, B. Harbourne, A. Lorenzini, {\em Inductively computable
unions of fat linear subspaces}, J. Pure Appl. Algebra. 219 (2015),
5413--5425.

\bibitem[Ful84]{Ful}
W. Fulton, {\em Intersection theory}, Springer-Verlag, Berlin
Heidelberg New York, 1984.

\bibitem[GMR83]{GMR}
A. V. Geramita, P. Maroscia, L. G. Roberts, {\em The Hilbert
function of a reduced k-algebra}, J. Lond. Math. Soc. (2) 28 (3)
(1983), 443--452.

\bibitem[GO81]{GO}
A. V. Geramita, F. Orecchia, {\em On the Cohen--Macaulay type of $s$
lines in $\mathbb{A}^{n+1}$}, J. Algebra. 70 (1981), 116--140.

\bibitem[Har04]{HR}
B. Harbourne,  J. Ro{\'e}, {\em Linear systems with multiple base
points in {$\Bbb P^2$}},  Adv. Geom. 4(1) (2004), 41--59.

\bibitem[HH81]{HH}
R. Hartshorne, A. Hirschowitz, {\em Droites en position g\'en\'erale
dans l'espace projectif}, in: Algebraic Geometry, La R\'abida,
(1981); in: Lecture Notes in Math, vol. 961, Springer, Berlin,
(1982), 169--188.

\bibitem[Hir81]{Hir}
A. Hirschowitz, {\em Sur la postulation g\'en\'erique des courbes
rationnelles}, Acta Math. 146 (1981), 209--230.


\end{thebibliography}
\end{document}